\documentclass{article}%
\usepackage{amsmath}%
\usepackage{amsfonts}%
\usepackage{amssymb}%
\usepackage{graphicx}
\newtheorem{theorem}{Theorem}

\newtheorem{definition}[theorem]{Definition}

\newtheorem{lemma}[theorem]{Lemma}

\newtheorem{proposition}[theorem]{Proposition}

\newenvironment{proof}[1][Proof]{\noindent\textbf{#1.} }{\ \rule{0.5em}{0.5em}}
\begin{document}

\title{A bracket polynomial for graphs, IV. \\Undirected Euler circuits, graph-links and multiply marked graphs}
\author{Lorenzo Traldi\\Lafayette College\\Easton, Pennsylvania 18042}
\date{}
\maketitle

\begin{abstract}
In earlier work we introduced the graph bracket polynomial of graphs with
marked vertices, motivated by the fact that the Kauffman bracket of a link
diagram $D$ is determined by a looped, marked version of the interlacement
graph associated to a directed Euler system of the universe graph of $D$. Here
we extend the graph bracket to graphs whose vertices may carry different kinds
of marks, and we show how multiply marked graphs encode interlacement with
respect to arbitrary (undirected) Euler systems. The extended machinery brings
together the earlier version and the graph-links of D. P. Ilyutko and V. O.
Manturov [\textit{J. Knot Theory Ramifications} \textbf{18} (2009), 791-823].
The greater flexibility of the extended bracket also allows for a recursive
description much simpler than that of the earlier version.

\bigskip

\textit{Keywords. }circuit partition, graph, graph-link, interlacement, Jones
polynomial, Kauffman bracket, local complement, Reidemeister move, virtual link

\bigskip

\textit{2000 Mathematics Subject Classification.} 57M25, 05C50

\end{abstract}

\section{Introduction}

An \textit{oriented link diagram} is a finite collection of oriented,
piecewise smooth closed curves in the plane, whose only singularities are
finitely many \textit{crossings} (double points). Classical crossings may be
positive or negative, as indicated in Fig. \ref{linksbis4}, and there may also
be virtual crossings. On the rare occasion when we want to restrict attention
to diagrams without virtual crossings, we refer to \textit{classical diagrams
}or \textit{classical links}.%
%TCIMACRO{\FRAME{ftbpFU}{4.3042in}{1.6215in}{0pt}{\Qcb{The smoothings of
%classical crossings.}}{\Qlb{linksbis4}}{linksbis4.ps}%
%width 4.3042in;  height 1.6215in;  depth 0pt;  original-width 8.246in;
%original-height 10.6969in;  cropleft "0.1622";  croptop "0.9121";
%cropright "0.8536";  cropbottom "0.7134";
%filename 'linksbis4.ps';file-properties "XNPEU";}}}%
%BeginExpansion
\begin{figure}
[ptb]
\begin{center}
\includegraphics[
trim=1.337501in 7.631169in 1.207214in 0.940257in,
height=1.6215in,
width=4.3042in
]%
{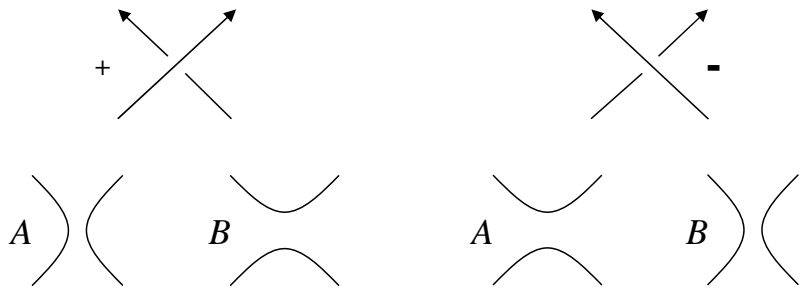}%
\caption{The smoothings of classical crossings.}%
\label{linksbis4}%
\end{center}
\end{figure}
%EndExpansion

The Kauffman bracket of an oriented link diagram $D$ is defined by a formula
that incorporates the numbers of closed curves in the various \textit{Kauffman
states} of $D$ \cite{Kau, Kv}. If $D$ has $n$ classical crossings then it has
$2^{n}$ states, obtained by choosing one of the two \textit{smoothings} at
each classical crossing; see Fig. \ref{linksbis4}. The bracket is then
\[
\lbrack D]=\sum_{S}A^{a(S)}B^{b(s)}d^{c(S)-1},
\]
in which the contribution of each state $S$ is determined by the number $a(S)
$ of $A$ smoothings in $S$, the number $b(S)$ of $B$ smoothings in $S$, and
the number $c(S)$ of closed curves in $S$. We use $[D]$ rather than the more
familiar notation $\left\langle D\right\rangle $ in order to distinguish this
three-variable bracket polynomial from its reduced one-variable form.

Associated to an oriented link diagram $D$ there is a 4-regular graph $U$, the
\textit{universe graph }of $D$, whose vertices correspond to the classical
crossings of $D$ and whose edges correspond to the arcs of $D$. Each vertex of
$U$ carries the sign of the corresponding crossing of $D$. The 2-in, 2-out
directed graph obtained from $U$ by directing its edges in accordance with the
orientations of the link components is denoted $\vec{U}$. It should be
emphasized that although $D$ is given as a specific subset of the plane, we
regard $U$ and $\vec{U}$ as abstract (nonimbedded) graphs, with signed
vertices. For the universe (di)graphs of two diagrams to be
\textit{isomorphic} (informally, \textit{the same}) there must be a one-to-one
correspondence between their vertices that preserves not only edges but also
vertex signs; the correspondence need not be compatible with the way the
diagrams are drawn in the plane.

If $D$ is a diagram of a knot $K$ then it has a \textit{looped interlacement
graph} $\mathcal{L}(D)$, i.e. the graph whose vertices correspond to the
classical crossings of $D$ and whose edges are defined by (a) a vertex is
looped if and only if the corresponding crossing is negative and (b) two
distinct vertices are adjacent if and only if the corresponding crossings are
\textit{interlaced} on $K$, i.e. when we follow $K$ around $D$ we encounter
first one of the two crossings, then the second, then the first again, and
then the second again. See Fig. \ref{linksbis7} for an example. (As usual, the
encircled crossing is virtual.)%
%TCIMACRO{\FRAME{ftbpFU}{4.7046in}{1.1208in}{0pt}{\Qcb{A trivial knot diagram
%$D$, the directed universe $\vec{U}$, the Euler circuit determined by the
%knot, and $\QTR{cal}{L}(D)$. When following the Euler circuit determined by
%the knot, we traverse each vertex without changing the pattern of
%dashes.}}{\Qlb{linksbis7}}{linksbis7.ps}%
%width 4.7046in;  height 1.1208in;  depth 0pt;  original-width 8.246in;
%original-height 10.6969in;  cropleft "0.1136";  croptop "0.8744";
%cropright "0.8697";  cropbottom "0.7381";
%filename 'linksbis7.ps';file-properties "XNPEU";}}}%
%BeginExpansion
\begin{figure}
[ptb]
\begin{center}
\includegraphics[
trim=0.936746in 7.895382in 1.074454in 1.343530in,
height=1.1208in,
width=4.7046in
]%
{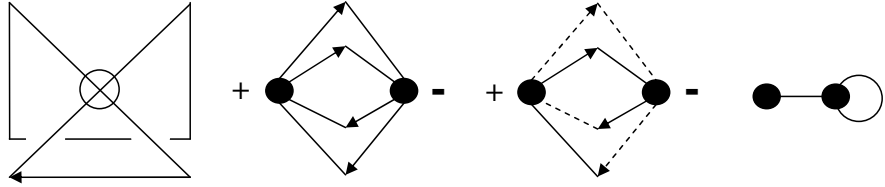}%
\caption{A trivial knot diagram $D$, the directed universe $\vec{U}$, the
Euler circuit determined by the knot, and $\mathcal{L}(D)$. When following the
Euler circuit determined by the knot, we traverse each vertex without changing
the pattern of dashes.}%
\label{linksbis7}%
\end{center}
\end{figure}
%EndExpansion

The looped interlacement graph was introduced in the first paper of this
series \cite{TZ}, where Zulli and the present author showed that if $D$ is a
classical or virtual knot diagram then $\mathcal{L}(D)$ contains enough
information about the states of $D$ to determine both the Kauffman bracket
$[D]$ and the Jones polynomial $V_{D}$ \cite{Jo}. As $\mathcal{L}(D)$ is
determined by the abstract graph $\vec{U}$ and the Euler circuit of $\vec{U}$
that corresponds to the diagrammed knot, this provides a striking conceptual
simplification of the Kauffman brackets and Jones polynomials of knots: if $D
$ and $D^{\prime}$ are knot diagrams with the same universe digraph $\vec
{U}=\overrightarrow{U^{\prime}}$, and the knots diagrammed in $D$ and
$D^{\prime}$ correspond to the same Euler circuit, then $[D]=[D^{\prime}]$ and
$V_{D}=V_{D^{\prime}}$. (To say the same thing in a different way: if two knot
diagrams $D$ and $D^{\prime}$ represent immersions in the plane of the same
abstract directed graph $\vec{U}$, with the same Euler circuit of $\vec{U}$
corresponding to both diagrammed knots and every pair of crossings
corresponding to the same vertex having the same sign, then $[D]=[D^{\prime}]$
and $V_{D}=V_{D^{\prime}}$.) For instance, in \cite{Kv} Kauffman mentioned
that the virtual knot diagram $D$ that appears at the top left in Fig.
\ref{nunknot} has $V_{D}=1$, even though $D$ is not a diagram of the unknot.
As indicated in the figure, its $\vec{U}$\ and $\mathcal{L}(D)$ are isomorphic
to those of an unknot diagram. Indeed, every classical or virtual knot diagram
with $V_{D}=1$ that we have seen has $\vec{U}$\ and $\mathcal{L}(D)$
isomorphic to those of an unknot diagram.%
%TCIMACRO{\FRAME{ftbphFU}{4.8075in}{2.1318in}{0pt}{\Qcb{A nontrivial knot
%diagram and a trivial knot diagram give rise to the same directed universe
%graph. The two knots determine the same Euler circuit and (hence) the same
%looped interlacement graph. }}{\Qlb{nunknot}}{linksbis101.ps}%
%width 4.8075in;  height 2.1318in;  depth 0pt;  original-width 8.246in;
%original-height 10.6969in;  cropleft "0.1294";  croptop "0.8624";
%cropright "0.9024";  cropbottom "0.6003";
%filename 'linksbis101.ps';file-properties "XNPEU";}}}%
%BeginExpansion
\begin{figure}
[ptbh]
\begin{center}
\includegraphics[
trim=1.067032in 6.421350in 0.804810in 1.471894in,
height=2.1318in,
width=4.8075in
]%
{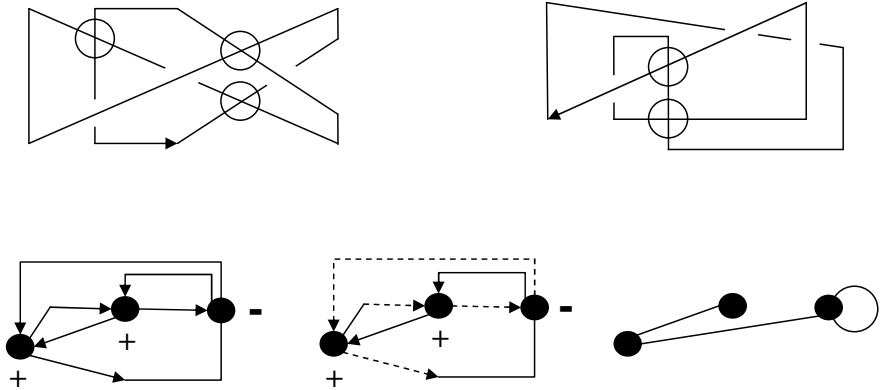}%
\caption{A nontrivial knot diagram and a trivial knot diagram give rise to the
same directed universe graph. The two knots determine the same Euler circuit
and (hence) the same looped interlacement graph. }%
\label{nunknot}%
\end{center}
\end{figure}
%EndExpansion

As discussed in \cite{TZ}, if we think of the Kauffman bracket $[D]$ as a
function of $\mathcal{L}(D)$ then this function may be extended to arbitrary
graphs. The extended function is called the \textit{graph bracket polynomial},
and it resembles the Kauffman bracket of virtual knot diagrams in several
ways, including the fact that it yields a graph Jones polynomial which is
invariant under graph operations suggested by the Reidemeister moves.

If $D$ is a diagram of an oriented multi-component link rather than a knot,
then $[D]$ can still be determined by using interlacement in $\vec{U}$, as
discussed in the second paper in this series \cite{T1}. The situation is
complicated by the fact that $U$ need not be connected, so interlacement is
defined with respect to a directed \textit{Euler system} $C$ of $\vec{U}$,
i.e. a set containing a directed Euler circuit for each connected component of
$\vec{U}$. Directed Euler systems certainly exist, but there is no canonical
way to choose a preferred one. Also, $D$ may contain link components that have
no classical crossings and hence are not detected by $U$; such link components
certainly affect $[D]$. These complications are handled by modifying $\vec{U}$
and $\mathcal{L}$ to incorporate additional information. First, $\vec{U}$ is
modified to include a \textit{free loop} corresponding to each link component
without any classical crossing. Free loops are essentially empty connected
components; they contain neither vertices nor edges but they contribute to
$c(U)$, the number of connected components of $U$. The modified interlacement
graph $\mathcal{L}(D$, $C)$ includes $c(U)-1$ free loops. The relationship
between $C$ and the link diagrammed in $D$ is recorded by \textit{marking} the
crossings of $D$ at which $C$ does not follow the incident link component(s);
the marks are transferred to the corresponding vertices of $\vec{U}$ and
$\mathcal{L}(D$, $C)$, and \textit{isomorphisms} are required to preserve free
loops, vertex marks and vertex signs. As before, the description of $[D]$ as a
function of $\mathcal{L}(D$, $C)$ extends directly to a bracket polynomial
defined for any graph that may include free loops and marked vertices, and
this bracket polynomial gives rise to a marked-graph Jones polynomial that is
invariant under the appropriate versions of the Reidemeister moves \cite{T1,
T3}.%
%TCIMACRO{\FRAME{ftbpFU}{4.1061in}{4.7357in}{0pt}{\Qcb{A marked link diagram
%$D$, the directed universe $\vec{U}$ with an Euler system $C$ indicated by
%dashes, and $\QTR{cal}{L}(D,C)$. The circuits of $C$ are consistent with the
%orientations of the link components. }}{\Qlb{linksbis2}}{linksbis2.ps}%
%width 4.1061in;  height 4.7357in;  depth 0pt;  original-width 8.246in;
%original-height 10.6969in;  cropleft "0.1619";  croptop "0.8998";
%cropright "0.8213";  cropbottom "0.3128";
%filename 'linksbis2.ps';file-properties "XNPEU";}}}%
%BeginExpansion
\begin{figure}
[ptb]
\begin{center}
\includegraphics[
trim=1.335027in 3.345990in 1.473560in 1.071829in,
height=4.7357in,
width=4.1061in
]%
{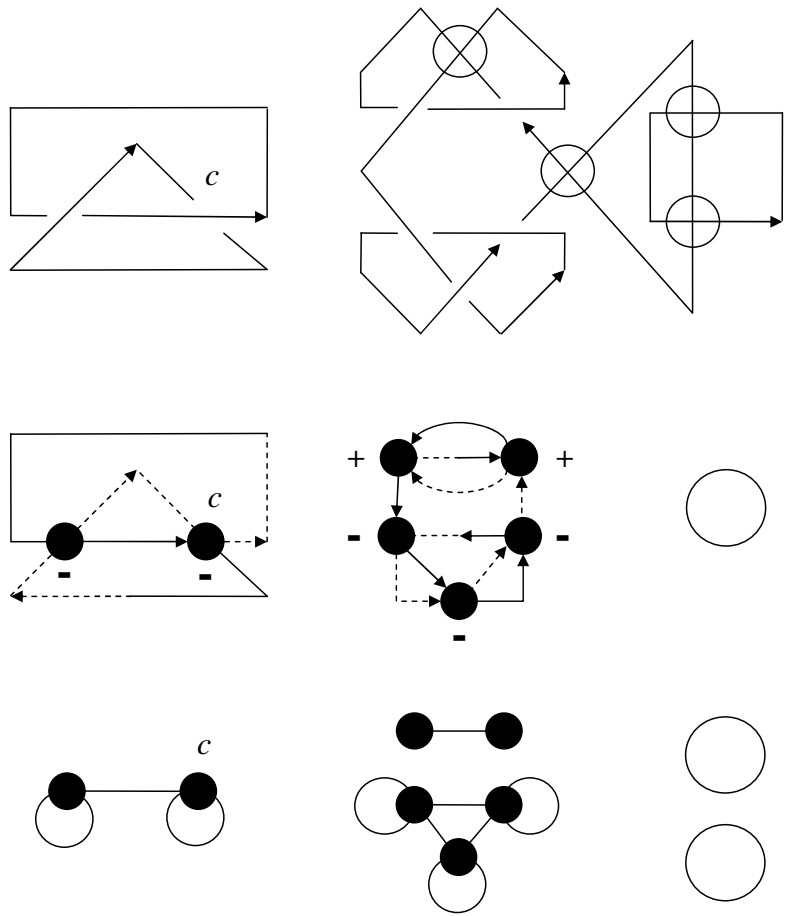}%
\caption{A marked link diagram $D$, the directed universe $\vec{U}$ with an
Euler system $C$ indicated by dashes, and $\mathcal{L}(D,C)$. The circuits of
$C$ are consistent with the orientations of the link components. }%
\label{linksbis2}%
\end{center}
\end{figure}
%EndExpansion

These definitions are illustrated in Fig. \ref{linksbis2}. The top row is a
diagram $D$ of an oriented link with four link components; virtual crossings
are encircled as usual. The mark $c$ on a crossing of $D$ specifies the Euler
system of $\vec{U}$ indicated by dashes. (That is, whenever we follow the
Euler circuit through a vertex we do not change the pattern of dashes.)
$\vec{U}$ is a 2-in, 2-out digraph with two nonempty connected components. One
nonempty connected component of $\vec{U}$ is pictured so as to resemble the
corresponding portion of $D$, and the other nonempty connected component is
not; as $\vec{U}$ is an abstract graph, we may picture it however we please.
$\vec{U}$ has a free loop corresponding to the crossing-less link component of
$D$; the free loop is indicated by a vertex-less circle. $\mathcal{L}(D$, $C)$
has looped vertices corresponding to negative crossings of $D$, and it has two
free loops because $c(U)=3$. The factors of the connected sum (the link's one
knotted component) are not interlaced, so the two nonempty connected
components of $U$ give rise to three nonempty connected components in
$\mathcal{L}(D$, $C)$. The only difference between $\mathcal{L}(D$, $C)$ and
the interlacement graph of a diagram $D^{\prime}$ obtained from $D$ by
splitting the connected sum into separate parts is that $\mathcal{L}%
(D^{\prime}$, $C)$ has more free loops.

Thistlethwaite \cite{Th} observed that the Kauffman bracket provides a
connection between knot theory (in particular, the Jones polynomials and
Kauffman brackets of classical links) and combinatorial theory (in particular,
the Tutte polynomial of planar graphs). Underlying Thistlethwaite's theorem is
a connection between circuit partitions of 4-regular plane graphs and Tutte
polynomials of their associated checkerboard graphs that was actually
discovered before the introduction of the Jones polynomial and Kauffman
bracket \cite{J1, L, Ma}. A useful technique in establishing this connection
involves giving a 4-regular plane graph an \textit{alternating orientation}
(or \textquotedblleft source-sink orientation\textquotedblright\ in the
terminology of \cite{M}), i.e. directing the edges so that the boundary of
each complementary region is coherently oriented, with (say) the boundaries of
white-colored regions oriented clockwise and the boundaries of black-colored
regions oriented counterclockwise. (The same technique has also been of use in
connection with the interlace polynomial of Arratia, Bollob\'{a}s and Sorkin
\cite{A2, A, EMS}.) If $D$ is a diagram of an oriented classical link then
clearly such a re-oriented version of $U$ is inconsistent with the link
components, in the sense that we cannot follow a link component through any
vertex without disregarding edge-directions. Directed Euler circuits of this
re-oriented version of $U$ have been called \textit{bent Euler tours}
\cite{GR}, \textit{rotating circuits} \cite{I, IM, IM1}, $\sigma
$\textit{-lines} \cite{K} and \textit{non-crossing Euler tours} \cite{L}.

Interlacement with respect to rotating circuits is the fundamental notion of
Ilyutko's and Manturov's theory of \textit{graph-links}, a relative of the
theory of looped interlacement graphs outlined above. Just as the looped
graphs considered in the first paper in this series were motivated by knot
diagrams, the first graph-links were motivated by link diagrams associated
with geometric structures called \textit{orientable atoms} \cite{IM}. Both
theories have grown more general since they were introduced, and may now be
used with arbitrary link diagrams; however they are not quite the same. The
graph-link theory is motivated by interlacement with respect to rotating
circuits \cite{IM1}, and the marked-graph theory, instead, is motivated by
interlacement with respect to Euler systems that respect the orientations of
the link components \cite{T1}. This difference is reflected in the fact that
the writhe of a marked graph is quite a simple notion -- just subtract the
number of looped vertices from the number of unlooped vertices -- while there
is no such simple notion of writhe for a graph-link. Ilyutko \cite{I} has
proven that the Reidemeister equivalence classes of graphs defined in
\cite{TZ} correspond precisely to \textit{graph-knots} (i.e. the graph-links
for which the sum of the adjacency matrix and an identity matrix is
invertible), but it is not clear whether or not this equivalence extends to
the general case.

Our purpose in the present paper is to extend the marked-graph machinery to
allow interlacement with respect to arbitrary Euler systems of $U$, thereby
developing a single theory that brings together graph-links and marked graphs.
In Section 2 we explain how to associate a marked interlacement graph
$\mathcal{L}(D,C)$ to an arbitrary Euler system $C$ in the universe graph of a
link diagram $D$; six different kinds of marks are used to record the
different ways an Euler system can pass through a crossing. The various marked
interlacement graphs that result from choosing different Euler systems in $D$
are related to each other through a marked version of local complementation,
the fundamental operation of the theory of circle graphs \cite{K, RR}. Our
vertex marks involve the letters $c$, $r$ and $u$, so we use $G_{cru}^{v}$ to
denote the marked local complement of $G$. In Section 3 we discuss
marked-graph versions of the Reidemeister moves. In Section 4 we define the
bracket polynomial of a marked graph, and show that it is invariant under
marked local complementation. If $D$ is a link diagram then the Kauffman
bracket $[D]$ is the same as the marked-graph bracket $[\mathcal{L}(D,C)]$.
The marked-graph version of the Jones polynomial \cite{Jo}, $V_{G}$, is
obtained from the bracket $[G]$ in the usual way, i.e. by evaluating $B\mapsto
A^{-1}$ and $d\mapsto-A^{2}-A^{-2}$, and then multiplying by a factor given by
the writhe and the number of vertices. The marked-graph Jones polynomial is
invariant under the marked-graph\ Reidemeister moves. In\ Sections 5 and 6 we
discuss the relationship between graph-links and marked graphs.

The Kauffman bracket of a (virtual) link diagram is recursively calculated by
eliminating classical crossings one at a time, applying the formula
$[D]=A[D_{A}]+B[D_{B}]$ at each step \cite{Kau, Kv}. The marked-graph bracket
polynomial of \cite{T1, TZ} is calculated using a recursive algorithm that is
considerably more complicated: different recursive steps are applied in
different circumstances, according to the placement of loops and marks. It
turns out, though, that using marked local complementation we can also devise
a much simpler algorithm, similar to that of the Kauffman bracket.

\begin{theorem}
\label{recursion} The marked-graph bracket polynomial of a marked graph $G$
can be calculated recursively using the following properties.

(a) The bracket polynomial of the empty graph is $[\emptyset]=1$, and the
bracket polynomial of a 1-vertex graph is given by the following.
\begin{gather*}
\lbrack(v)]=[(v,r,\ell)]=Ad+B=[(v,u)]=[(v,ur,\ell)]\\
\lbrack(v,\ell)]=[(v,r)]=A+Bd=[(v,u,\ell)]=[(v,ur)]\\
\lbrack(v,c)]=[(v,cr,\ell)]=A+B=[(v,c,\ell)]=[(v,cr)]
\end{gather*}
Here $(v)$ indicates that $v$ is unlooped and unmarked, $(v,\ell)$ indicates
that $v$ is looped and unmarked, $(v,cr)$ indicates that $v$ is unlooped and
marked $cr$, $(v,u,\ell)$ indicates that $v$ is looped and marked $u$, and so on.

(b) If $G^{\prime}$ is obtained from $G$ by removing a free loop then
$[G]=d\cdot\lbrack G^{\prime}]$.

(c) If $G^{\prime}$ is obtained from $G$ by removing an isolated vertex $v$
then $[G]=[\{v\}]\cdot\lbrack G^{\prime}]$; $[\{v\}]$ is given in part (a).

(d) If the vertex $v$ is unlooped and marked $c$, or looped and marked $cr$,
then
\[
\lbrack G]=A[G-v]+B[G_{cru}^{v}-v],
\]
where $G_{cru}^{v}$ is the marked local complement of $G$ with respect to $v$.
On the other hand, if $v$ is looped and marked $c$, or unlooped and marked
$cr$, then%
\[
\text{\ }[G]=B[G-v]+A[G_{cru}^{v}-v].
\]

(e) If $v$ has a neighbor marked $u$ or $ur$ then $[G]=[G_{cru}^{v}]$.

(f) If the vertex $w$ is unmarked or marked $r$ then $[G]=[G_{cru}^{w}]$.
\end{theorem}

Despite having six options, Theorem \ref{recursion} is quite similar to the
Kauffman bracket's recursion. Parts (a) -- (c) correspond to simple properties
involving very small portions of link diagrams, and part (d) corresponds to
the formula $[D]=A[D_{A}]+B[D_{B}]$. At first glance parts (e) and (f) may
seem to be novel complications, but it is important to remember that Theorem
\ref{recursion} is applied to abstract graphs and the Kauffman bracket,
instead, is applied to plane diagrams. When we draw $D$ in the plane, we know
which transition to call $A$ and which transition to call $B$ at each
crossing; and when we smooth a crossing to obtain $D_{A}$ and $D_{B}$, these
new diagrams are drawn in the plane so as to guarantee appropriate choices of
$A$ and $B$ smoothings at the remaining crossings. Similarly, the use of
marked local complementations in (e) and (f), together with the restriction of
(d) to vertices marked $c$ or $cr$, ensures that when we replace an abstract
graph $G$ with smaller abstract graphs during the recursion, the smaller
graphs inherit the appropriate $A$ and $B$ assignments.

At the end of the paper we briefly discuss appropriate modifications of the
results of \cite{T3}, involving the use of vertex weights in streamlining
bracket calculations.%

%TCIMACRO{\FRAME{ftbpFU}{1.7469in}{1.7608in}{0pt}{\Qcb{A chord diagram
%representing the Euler system of Fig. \ref{linksbis2}.}}{\Qlb{linksbis102}%
%}{linksbis102.ps}{\special{ language "Scientific Word";  type "GRAPHIC";
%width 1.7469in;  height 1.7608in;  depth 0pt;  original-width 8.246in;
%original-height 10.6969in;  cropleft "0.2917";  croptop "0.8496";
%cropright "0.7082";  cropbottom "0.5253";
%filename 'linksbis102.ps';file-properties "XNPEU";}}}%
%BeginExpansion
\begin{figure}
[ptb]
\begin{center}
\includegraphics[
trim=2.405358in 5.619082in 2.406183in 1.608814in,
height=1.7608in,
width=1.7469in
]%
{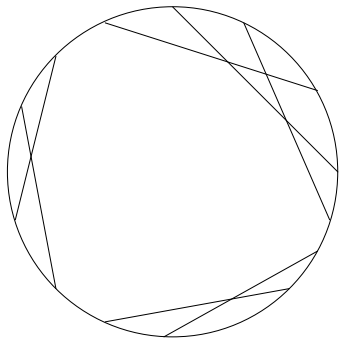}%
\caption{A chord diagram representing the Euler system of Fig. \ref{linksbis2}%
.}%
\label{linksbis102}%
\end{center}
\end{figure}
%EndExpansion

Before beginning a detailed discussion we recall that Euler systems of
4-regular graphs are equivalent to two other familiar combinatorial
structures: double occurrence words and chord diagrams. For instance, the
Euler system of Fig. \ref{linksbis2} could be represented by the word
$bcbcdefghfghde$, or by the chord diagram in Fig. \ref{linksbis102}. (In order
to carry as much information as Fig. \ref{linksbis2} does, the double
occurrence word or chord diagram should incorporate the vertex signs and
marks.) Although these three kinds of structures are equivalent to each other,
we prefer to use Euler systems in 4-regular graphs. Our first reason for this
preference is the obviousness of the observation that a typical 4-regular
graph has many different Euler systems; the equivalence relations on double
occurrence words and chord diagrams motivated by this obvious observation are
not so intuitively immediate. Our second reason is the richness of the
combinatorial theory of 4-regular graphs, which has been developed by Bouchet,
Jaeger, Las Vergnas, Martin and others in the decades since Kotzig's
foundational work \cite{K}. We believe this beautiful theory will prove to be
of great interest to knot theorists. In particular, Bouchet's comment that
\textquotedblleft the theory of isotropic systems is the theory of simple
graphs up to local complementations\textquotedblright\ \cite{Bgiso} makes it
seem likely that much of our machinery could be re-cast using marked versions
of isotropic systems or multimatroids \cite{B}.

\section{Interlacement and local complementation}%

%TCIMACRO{\FRAME{ftbpFU}{4.8066in}{2.5278in}{0pt}{\Qcb{The six ways an\ Euler
%system might be related to a link at a crossing.}}{\Qlb{linksbis5}%
%}{linksbis5.ps}{\special{ language "Scientific Word";  type "GRAPHIC";
%width 4.8066in;  height 2.5278in;  depth 0pt;  original-width 8.246in;
%original-height 10.6969in;  cropleft "0.1133";  croptop "0.9246";
%cropright "0.8859";  cropbottom "0.6130";
%filename 'linksbis5a.ps';file-properties "XNPEU";}}}%
%BeginExpansion
\begin{figure}
[ptb]
\begin{center}
\includegraphics[
trim=0.934272in 6.557199in 0.940868in 0.806546in,
height=2.5278in,
width=4.8066in
]%
{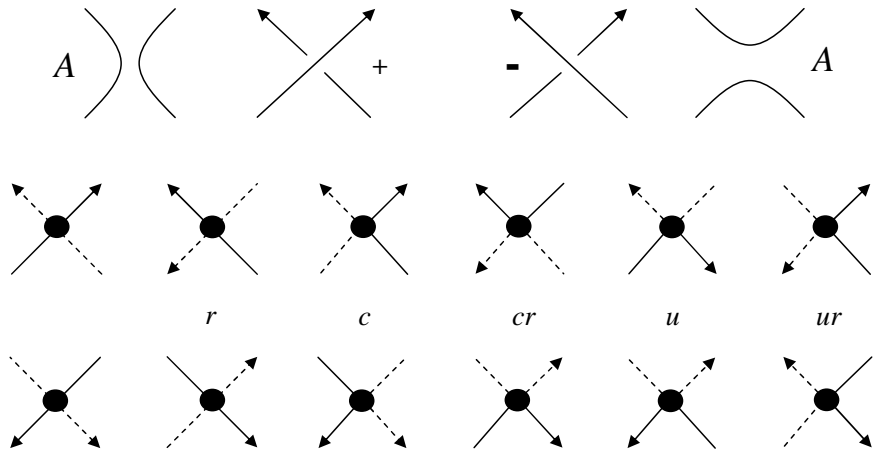}%
\caption{The six ways an\ Euler system might be related to a link at a
crossing.}%
\label{linksbis5}%
\end{center}
\end{figure}
%EndExpansion

Suppose $D$ is an oriented link diagram, $U$ is the undirected universe graph
and $C$ is an arbitrary\ Euler system of $U$, i.e. $C$ contains one Euler
circuit for each connected component of $U$. At each vertex there are six
different ways the incident circuit of an undirected Euler system $C$ might be
related to the incident link component(s). See Fig. \ref{linksbis5}. Note that
the edge-directions in the figure do not agree: those in the top row refer to
the orientations of the link components, and those in the two lower rows refer
to an orientation of the incident circuit of $C$. We do not regard the
circuits of an Euler system as carrying preferred orientations, so each
circuit can be oriented in either of the two possible ways. Consequently the
two lower rows of Fig. \ref{linksbis5} picture six cases, not twelve; the two
in each column are the same. The six cases fall into three pairs, indicated by
the letter $r$ (for \textit{rotate}).

These considerations motivate the following definitions.

\begin{definition}
\label{mgraph}A graph is \emph{multiply marked} by assigning to each of its
vertices one of the six labels of Figure \ref{linksbis5}.
\end{definition}

We often use \textit{marked} rather than \textit{multiply marked}; when we
want to focus attention on the special cases considered in \cite{T1, T3, TZ}
we say \textit{unmarked} or \textit{singly marked}.

\begin{definition}
\label{mint} Let $D$ be an oriented link diagram, and let $C$ be an Euler
system $C$ of the universe graph $U$. The vertices of $U$ are assigned marks
and signs as in Figure \ref{linksbis5}, and the\emph{\ marked interlacement
graph} $\mathcal{L}(D$, $C)$ is defined as follows.

1. Vertices correspond to classical crossings of $D$. They are assigned marks
as in Fig. \ref{linksbis5}.

2. A vertex is looped if and only if the corresponding crossing is negative$.
$

3. Two distinct vertices are adjacent if and only if they are interlaced with
respect to $C$.

4. $\mathcal{L}(D$, $C)$ has $c(U)-1$ free loops.
\end{definition}

\begin{definition}
The \emph{writhe} of a graph $G$ with $n$ vertices and $\ell$ looped vertices
is $w(G)=n-2\ell$.
\end{definition}

Fig. \ref{linksbis3} shows the result of applying these definitions to the
example of Fig. \ref{linksbis2}, with a different Euler system. To keep the
figure simple, the signs of the negative vertices of $U$ are not indicated.
Note also that the arc-directions in $D$ reflect the orientations of the link
components, but the edge-directions in $U$ reflect a choice of orientations
for the Euler circuits in $C$.%

%TCIMACRO{\FRAME{ftbpFU}{4.1087in}{5.0401in}{0pt}{\Qcb{A marked link diagram
%$D$, the universe $U$ with an Euler system $C$ indicated by dashes, and
%$\QTR{cal}{L}(D,C)$. The edge-directions and patterns of dashes in $U$
%indicate walks along the circuits of $C$; they are not consistent with the
%link components. To reduce clutter, signs are indicated only for the two
%positive vertices of $U$.}}{\Qlb{linksbis3}}{linksbis3a.ps}%
%width 4.1087in;  height 5.0401in;  depth 0pt;  original-width 8.246in;
%original-height 10.6969in;  cropleft "0.1612";  croptop "0.8996";
%cropright "0.8208";  cropbottom "0.2751";
%filename 'linksbis3a.ps';file-properties "XNPEU";}}}%
%BeginExpansion
\begin{figure}
[ptb]
\begin{center}
\includegraphics[
trim=1.329255in 2.942717in 1.477683in 1.073969in,
height=5.0401in,
width=4.1087in
]%
{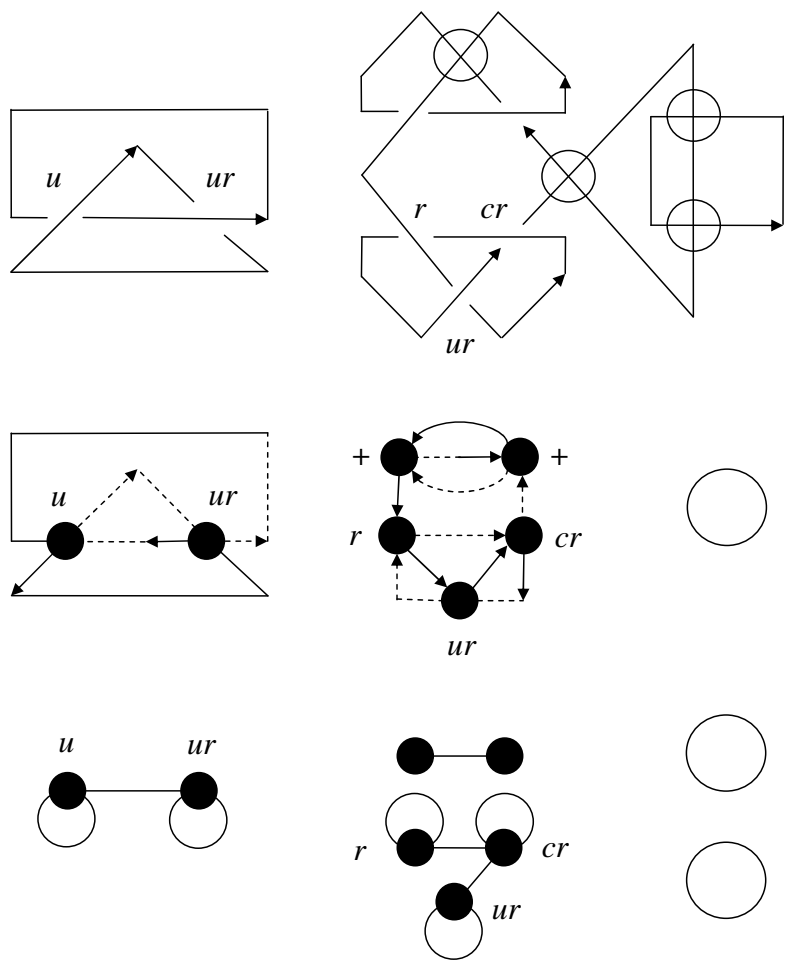}%
\caption{A marked link diagram $D$, the universe $U$ with an Euler system $C$
indicated by dashes, and $\mathcal{L}(D,C)$. The edge-directions and patterns
of dashes in $U$ indicate walks along the circuits of $C$; they are not
consistent with the link components. To reduce clutter, signs are indicated
only for the two positive vertices of $U$.}%
\label{linksbis3}%
\end{center}
\end{figure}
%EndExpansion

If $C$ is an Euler system of $G$ and $v\in V(G)$ then Kotzig \cite{K} defined
the $\kappa$\textit{-transform} $C\ast v$ to be the Euler system obtained from
$C$ by reversing one of the two $v$-to-$v$ paths in the Euler circuit of $C$
incident on $v$, and he proved that the various Euler systems of a 4-regular
graph $G$ are all related to each other through $\kappa$-transformations. (A
proof appears in \cite{T1}.) We do not regard Euler systems as carrying
preferred orientations, so it does not matter which of the two $v$-to-$v$
paths is reversed. The effect of a $\kappa$-transformation on interlacement is
easy to see: the only interlacements that are changed are those that involve
two vertices both of which appear precisely once on each $v$-to-$v$ path of
$C$, i.e. both of which are interlaced with $v$; the effect of the $\kappa
$-transformation is to toggle (reverse) the interlacement of every such pair.
Consequently, the effect of a $\kappa$-transformation $C\ast v$ on vertices of
the interlacement graph other than than $v$ itself is partly described by
\textit{local complementation}.%

%TCIMACRO{\FRAME{ftbpFU}{3.1012in}{0.6036in}{0pt}{\Qcb{$C$ and $C\ast v$.}%
%}{\Qlb{linksbis8}}{linksbis8.ps}{\special{ language "Scientific Word";
%type "GRAPHIC";  maintain-aspect-ratio TRUE;  display "USEDEF";
%valid_file "F";  width 3.1012in;  height 0.6036in;  depth 0pt;
%original-width 8.246in;  original-height 10.6969in;  cropleft "0.2917";
%croptop "0.8934";  cropright "0.7886";  cropbottom "0.8214";
%filename 'linksbis8.ps';file-properties "XNPEU";}}}%
%BeginExpansion
\begin{figure}
[ptb]
\begin{center}
\includegraphics[
trim=2.405358in 8.786434in 1.743204in 1.140289in,
height=0.6036in,
width=3.1012in
]%
{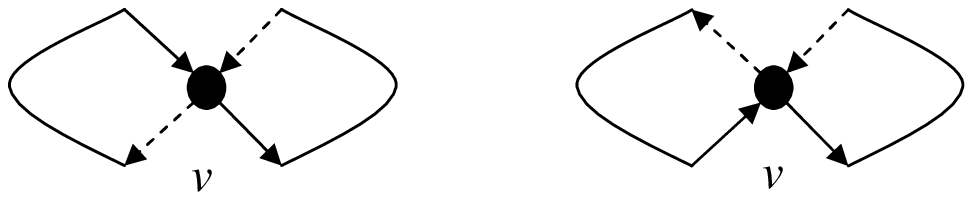}%
\caption{$C$ and $C\ast v$.}%
\label{linksbis8}%
\end{center}
\end{figure}
%EndExpansion

\begin{definition}
\label{lc} If $G$ is a graph and $v\in V(G)$ then the \emph{local complement}
$G^{v}$ is the graph obtained from $G$ by toggling edges involving only
neighbors of $v$. That is, if $w\neq v$ is adjacent to $v$ then $w$ is looped
in $G^{v}$ if and only if it is not looped in $G$; and if $v\neq x\neq y\neq
v$ and $x$, $y\,$\ are both neighbors of $v$ then $x$, $y$ are adjacent in
$G^{v}$ if and only if they are not adjacent in $G$.
\end{definition}

Observe that this definition involves changes to both loops and non-loop
edges. A different definition, which is intended for simple graphs and
consequently affects only non-loop edges, also appears in the combinatorial literature.

Read and Rosenstiehl \cite{RR} noted that for unlooped, unmarked interlacement
graphs, simple local complementation at $v$ completely describes the effect of
a $\kappa$-transformation at $v$. For us, however, this is not quite true,
because local complementation does not have the correct effect on the vertex
marks of neighbors of $v$: if $v$ and $w$ are interlaced with respect to $C$
then in $C\ast v$ the direction of one passage through $w$ is reversed, and
looking at Fig. \ref{linksbis5} we see that this changes the vertex-mark of
$w$ according to the pairings $ur\leftrightarrow c$, $u\leftrightarrow cr$,
$r\leftrightarrow$(unmarked). Moreover, a $\kappa$-transformation at $v$
affects the mark of $v$ itself, as illustrated in Fig. \ref{linksbis6}. Taking
these effects into account, we are led to the next definition.%

%TCIMACRO{\FRAME{ftbpFU}{4.0006in}{2.0764in}{0pt}{\Qcb{The effect of a $\kappa
%$-transformation on the vertex where it is performed.}}{\Qlb{linksbis6}%
%}{linksbis6.ps}{\special{ language "Scientific Word";  type "GRAPHIC";
%width 4.0006in;  height 2.0764in;  depth 0pt;  original-width 8.246in;
%original-height 10.6969in;  cropleft "0.1777";  croptop "0.8632";
%cropright "0.8204";  cropbottom "0.6093";
%filename 'linksbis6.ps';file-properties "XNPEU";}}}%
%BeginExpansion
\begin{figure}
[ptb]
\begin{center}
\includegraphics[
trim=1.465314in 6.517622in 1.480982in 1.463336in,
height=2.0764in,
width=4.0006in
]%
{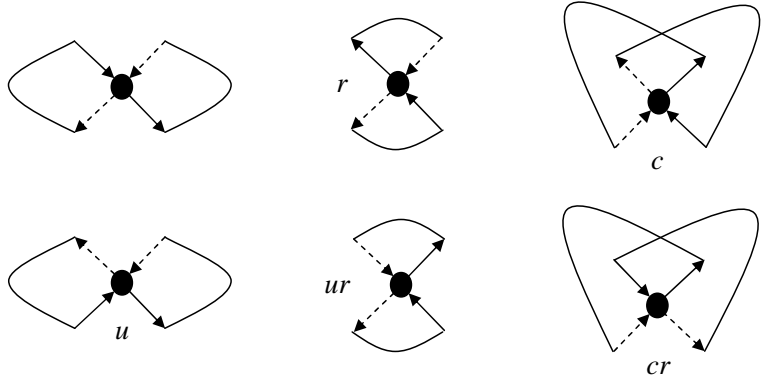}%
\caption{The effect of a $\kappa$-transformation on the vertex where it is
performed.}%
\label{linksbis6}%
\end{center}
\end{figure}
%EndExpansion

\begin{definition}
\label{mlc} If $G$ is a marked graph and $v\in V(G)$ then the\emph{\ marked
local complement} $G_{cru}^{v}$ is the graph obtained from $G$ by making the
following changes, and no others.

1. If $v$ is unmarked in $G$ then it is marked $u$ in $G_{cru}^{v}$, and vice versa.

2. If $v$ is marked $r$ in $G$ then it is marked $ur$ in $G_{cru}^{v}$, and
vice versa.

3. If $v$ is marked $c$ in $G$ then it is marked $cr$ in $G_{cru}^{v}$, and
vice versa.

4. If $w\neq v$ is an unmarked neighbor of $v$ in $G$ then $w$ is marked $r$
in $G_{cru}^{v}$, and vice versa.

5. If $w\neq v$ is a neighbor of $v$ marked $c$ in $G$, then $w$ is marked
$ur$ in $G_{cru}^{v}$, and vice versa.

6. If $w\neq v$ is a neighbor of $v$ marked $u$ in $G$, then $w$ is marked
$cr$ in $G_{cru}^{v}$, and vice versa.

7. If $v\neq x\neq y\neq v$ and $x$, $y\,$\ are both neighbors of $v$ then $x
$, $y$ are adjacent in $G_{cru}^{v}$ if and only if they are not adjacent in
$G$.
\end{definition}

Note that unlike Definition \ref{lc}, Definition \ref{mlc} does not involve
any loop-toggling, and consequently $w(G)=w(G_{cru}^{v})$. Definition
\ref{mlc} is the culmination of a rather long process of understanding the
effect on interlacement of changing Euler systems in link diagrams. \cite{TZ}
did not require changing Euler circuits at all, and \cite{IM} and \cite{T1}
both required some changing of Euler circuits, but could use appropriate
modifications of the more specialized \textit{pivot} operation. (As discussed
below, a pivot is expressible as a composition of local complementations; the
reverse is not true in general.) It was the appearance of a modified local
complement operation in \cite{IM1} that inspired the approach we take here.

If $D$ is a link diagram then Kotzig's theorem \cite{K} tells us that all the
Euler systems of $U$ are related to each other through $\kappa$%
-transformations. As the marked interlacement graph of $\mathcal{L}(D$, $C\ast
v)$ is the marked local complement $\mathcal{L}(D$, $C)_{cru}^{v}$, we
conclude that the marked interlacement graphs of $D$ are all related to each
other through marked local complementations.

\begin{theorem}
\label{diagramcomplement} Let $D$ be an oriented link diagram with a marked
interlacement graph $G=\mathcal{L}(D$, $C)$, and suppose $G^{\prime}$ is an
arbitrary marked graph. Then $G^{\prime}=\mathcal{L}(D$, $C^{\prime})$ for
some Euler system $C^{\prime}$ of $D$ if and only if $G^{\prime}$ can be
obtained from $G$ through marked local complementations.
\end{theorem}

We close this section by extending the marked pivot operation of \cite{T1} to
multiply marked graphs.

\begin{lemma}
\label{pivot} Suppose $G$ is a marked graph with two adjacent vertices $v\neq
w$. Let $N_{v}$ be the set of neighbors of $v$ that are not neighbors of $w$,
$N_{w}$ the set of neighbors of $w$ that are not neighbors of $v$, and
$N_{vw}$ the set of neighbors shared by $v$ and $w$; in particular, $v\in
N_{w}$ and $w\in N_{v}$.\ Then $((G_{cru}^{v})_{cru}^{w})_{cru}^{v}$ is the
graph obtained from $G$ by making the following changes, and no others.

(a) The mark on $v$ is changed according to the pattern $c\leftrightarrow
$(unmarked), $r\leftrightarrow cr$, $u\leftrightarrow ur$.

(b) The mark of $w$ is changed according to the same pattern.

(c) The neighbors of $v$ in $((G_{cru}^{v})_{cru}^{w})_{cru}^{v}$ are the
elements of $(N_{w}-v)\cup\{w\}\cup N_{vw}$ and the neighbors of $w$ in
$((G_{cru}^{v})_{cru}^{w})_{cru}^{v}$ are the elements of $(N_{v}%
-w)\cup\{v\}\cup N_{vw}$.

(d) Every adjacency involving two vertices from different elements of
$\{N_{v}-w$, $N_{w}-v$, $N_{vw}\}$ is toggled.
\end{lemma}

\begin{proof}
Definition \ref{mlc} tells us that the three local complementations affect the
mark of $v$ as follows: $c\leftrightarrow cr\leftrightarrow u\leftrightarrow
$(unmarked), $r\leftrightarrow ur\leftrightarrow c\leftrightarrow cr$ and
$u\leftrightarrow$(unmarked)$\leftrightarrow r\leftrightarrow ur$. The three
local complementations affect the mark of $w$ as follows: $c\leftrightarrow
ur\leftrightarrow r\leftrightarrow$(unmarked), $r\leftrightarrow
$(unmarked)$\leftrightarrow u\leftrightarrow c$, and $u\leftrightarrow
cr\leftrightarrow c\leftrightarrow ur$.

Part (c) is verified as follows. Observe first that the neighbor-sets of
vertices outside $N_{v}\cup N_{w}\cup N_{vw}$ are not affected by local
complementations at $v$ and $w$, and $v$ and $w$ remain neighbors through all
three local complementations. If $x\in N_{v}-w$ then $x$ is adjacent to both
$v$ and $w$ in $G_{cru}^{v}$, so $x$ is adjacent to $w$ and not $v$ in
$(G_{cru}^{v})_{cru}^{w}$; this remains the same in $((G_{cru}^{v})_{cru}%
^{w})_{cru}^{v}$. If $x\in N_{w}-v$ then $x$ is adjacent to $w$ and not $v$ in
$G_{cru}^{v}$, so $x$ is adjacent to both $v$ and $w$ in $(G_{cru}^{v}%
)_{cru}^{w}$, so $x$ is adjacent to $v$ and not $w$ in $((G_{cru}^{v}%
)_{cru}^{w})_{cru}^{v}$. If $x\in N_{vw}$ then $x$ is adjacent to $v$ and not
$w$ in $G_{cru}^{v}$, so $x$ is adjacent to $v$ and not $w$ in $(G_{cru}%
^{v})_{cru}^{w}$, so $x$ is adjacent to $v$ and $w$ in $((G_{cru}^{v}%
)_{cru}^{w})_{cru}^{v}$.

For part (d), if $x\in N_{v}-w$ and $y\in N_{w}-v$ then the adjacency between
$x$ and $y$ is unchanged in $G_{cru}^{v}$, then toggled in $(G_{cru}%
^{v})_{cru}^{w}$, and then unchanged in $((G_{cru}^{v})_{cru}^{w})_{cru}^{v}$.
If $x\in N_{v}-w$ and $y\in N_{vw}$ then the adjacency between $x$ and $y$ is
toggled in $G_{cru}^{v}$, then unchanged in $(G_{cru}^{v})_{cru}^{w}$, and
then unchanged in $((G_{cru}^{v})_{cru}^{w})_{cru}^{v}$. If $x\in N_{w}-v$ and
$y\in N_{vw}$ then the adjacency between $x$ and $y$ is unchanged in
$G_{cru}^{v}$, unchanged in $(G_{cru}^{v})_{cru}^{w}$, and then toggled in
$((G_{cru}^{v})_{cru}^{w})_{cru}^{v}$.

It remains to verify that no other change is made. If $x\notin N_{v}\cup
N_{w}\cup N_{vw}$ then none of the local complementations affects the mark of
$x$, or any adjacency involving $x$. If $x\neq y$ are in the same one of
$N_{v}-w$, $N_{w}-v$, $N_{vw}$ then their adjacency is toggled by two of the
three local complementations, so it remains unchanged in $((G_{cru}^{v}%
)_{cru}^{w})_{cru}^{v}$. If $x\in N_{v}-w$ then its mark is affected as
follows: $c\rightarrow ur\rightarrow c\rightarrow c$, $cr\rightarrow
u\rightarrow cr\rightarrow cr$, $u\rightarrow cr\rightarrow u\rightarrow u$,
$ur\rightarrow c\rightarrow ur\rightarrow ur$, $r\rightarrow$%
(unmarked)$\rightarrow r\rightarrow r$ , and (unmarked)$\rightarrow
r\rightarrow$(unmarked)$\rightarrow$(unmarked). If $x\in N_{w}-v$ then its
mark is affected as follows: $c\rightarrow c\rightarrow ur\rightarrow c$,
$cr\rightarrow cr\rightarrow u\rightarrow cr$, $u\rightarrow u\rightarrow
cr\rightarrow u$, $ur\rightarrow ur\rightarrow c\rightarrow ur$, $r\rightarrow
r\rightarrow$(unmarked)$\rightarrow r$, and (unmarked)$\rightarrow
$(unmarked)$\rightarrow r\rightarrow$(unmarked). Finally, if $x\in N_{vw}$
then its mark is affected as follows: $c\rightarrow ur\rightarrow
ur\rightarrow c$, $cr\rightarrow u\rightarrow u\rightarrow cr$, $u\rightarrow
cr\rightarrow cr\rightarrow u$, $ur\rightarrow c\rightarrow c\rightarrow ur$,
$r\rightarrow$(unmarked)$\rightarrow$(unmarked)$\rightarrow r$, and
(unmarked)$\rightarrow r\rightarrow r\rightarrow$(unmarked).
\end{proof}

\begin{definition}
\label{mpivot}Let $G$ be a doubly marked graph with two adjacent vertices
$v\neq w$. Then the graph $((G_{cru}^{v})_{cru}^{w})_{cru}^{v}$ is the
\emph{marked pivot} of $G$ with respect to $v$ and $w$, denoted $G_{cru}%
^{vw}.$
\end{definition}

The unmarked version of Definition \ref{mpivot} is the equality $((G^{v}%
)^{w})^{v}=G^{vw}$ relating local complements and pivots. This equality is a
familiar part of the theory of local complementation; see for instance
\cite{A, B}. As mentioned in \cite{A}, the unmarked version of part (c) of
Lemma \ref{pivot} is unnecessary; up to isomorphism, simply exchanging of the
names of $v$ and $w$ has the same effect. We include (c) because omitting it
would require more complicated versions of (a) and (b).

\section{Reidemeister equivalence}%

%TCIMACRO{\FRAME{fhFU}{3.8052in}{4.1355in}{0pt}{\Qcb{Three classical
%Reidemeister moves above a detour move.}}{\Qlb{flyfig2}}{flyfig2.ps}%
%width 3.8052in;  height 4.1355in;  depth 0pt;  original-width 8.246in;
%original-height 10.6969in;  cropleft "0.1780";  croptop "0.9122";
%cropright "0.7889";  cropbottom "0.4003";
%filename 'flyfig2.ps';file-properties "XNPEU";}}}%
%BeginExpansion
\begin{figure}
[h]
\begin{center}
\includegraphics[
trim=1.467788in 4.281969in 1.740731in 0.939188in,
height=4.1355in,
width=3.8052in
]%
{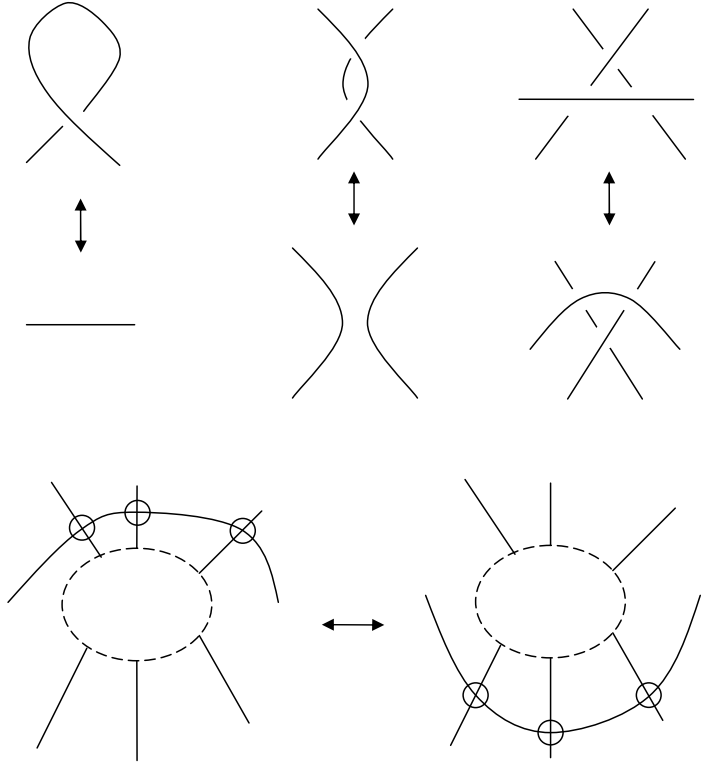}%
\caption{Three classical Reidemeister moves above a detour move.}%
\label{flyfig2}%
\end{center}
\end{figure}
%EndExpansion

Recall that diagrams representing the same virtual link type are obtained from
each other by using both \textit{classical} Reidemeister moves that involve
only classical crossings, and \textit{virtual} Reidemeister moves that involve
virtual crossings. As noted in \cite{FKM}, the virtual Reidemeister moves may
be subsumed in the more general \textit{detour} move: any arc containing no
classical crossing may be replaced by any other arc with the same endpoints,
provided that the only singularities on the new arc are finitely many double
points, and these double points are all designated as virtual crossings. See
Fig. \ref{flyfig2}. It is obvious that detour moves on $D$ have no effect on
$\mathcal{L}(D,C)$.

The effects of classical Reidemeister moves on singly marked interlacement
graphs were described in \cite{T1, TZ}, using an elegant idea due to
\"{O}stlund \cite{O}: explicit descriptions of all possible moves are not
required, so long as we describe sufficiently many moves to generate the rest
through composition.

The first kind of Reidemeister move from \cite{T1} involves adjoining or
deleting an unmarked, isolated vertex; the vertex may be looped or unlooped.
Using marked local complementation, an unmarked isolated vertex is transformed
into an isolated vertex marked $u$. It is not possible to obtain an isolated
vertex with any other mark. This reflects the fact that there are only two
ways an Euler circuit can traverse a trivial crossing in a link diagram; see
Fig. \ref{linksbis106}.%

%TCIMACRO{\FRAME{ftbpFU}{2.8997in}{0.921in}{0pt}{\Qcb{The two ways an Euler
%circuit can traverse a trivial crossing.}}{\Qlb{linksbis106}}{linksbis106.ps}%
%width 2.8997in;  height 0.921in;  depth 0pt;  original-width 8.246in;
%original-height 10.6969in;  cropleft "0.1298";  croptop "0.8746";
%cropright "0.5942";  cropbottom "0.7633";
%filename 'linksbis106.ps';file-properties "XNPEU";}}}%
%BeginExpansion
\begin{figure}
[ptb]
\begin{center}
\includegraphics[
trim=1.070331in 8.164944in 3.346227in 1.341391in,
height=0.921in,
width=2.8997in
]%
{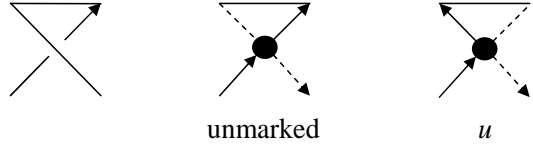}%
\caption{The two ways an Euler circuit can traverse a trivial crossing.}%
\label{linksbis106}%
\end{center}
\end{figure}
%EndExpansion

\begin{definition}
\label{R1} An $\Omega.1$ move is performed by adjoining or removing an
isolated vertex whose mark does not involve $c$ or $r$. The vertex may be
looped or unlooped.
\end{definition}

Four kinds of $\Omega.2$ moves are explicitly described in \cite{T1}.

\begin{definition}
\label{R2} Suppose $G$ is a marked graph with two vertices $v$ and $w$, $v$
looped and $w$ not looped. Then any of the following is an $\Omega
.2$\emph{\ move}, and so is the inverse transformation.

(a) Suppose $v$ and $w$ are both unmarked, and they have the same neighbors
outside $\{v$, $w\}$. Replace $G$ with $G-v-w$.

(b) Suppose $v$ is marked $c$, $w$ is unmarked, $v$ is the only neighbor of
$w$, and $z\notin\{v$, $w\}$ is a neighbor of $v$. Replace $G$ with
$G_{cru}^{vz}-v-w$.

(c) Suppose $v$ is marked $c$, $w$ is unmarked, $v$ and $w$ have the same
neighbors outside $\{v$, $w\}$, and $z\notin\{v$, $w\}$ is a neighbor of $v$
and $w$. Replace $G$ with $G_{cru}^{vz}-v-w$.

(d) Suppose $v$ is marked $c$, $w$ is unmarked, $v$ is the only neighbor of
$w$, and $w$ is the only neighbor of $v$. Replace $G$ with $G^{+}-v-w$, where
$G^{+}$ is obtained from $G$ by adjoining a free loop.
\end{definition}

Only one kind of $\Omega.3$ move is explicitly described in \cite{T1}.

\begin{definition}
\label{R3} \ Suppose $G$ is a marked graph with three unmarked vertices $u$,
$v$, $w$ such that $u$, $v$, $w$ are all adjacent to each other, $u$ is
looped, $v$ and $w$ are unlooped, and every vertex $x\notin\{u$, $v$, $w\}$ is
adjacent to either none or precisely two of $u$, $v$, $w$. An $\Omega.3$ move
is performed by replacing $G$ with the graph obtained by removing all three
edges $\{u$, $v\}$, $\{u$, $w\}$ and $\{v$, $w\}$.
\end{definition}

The inverse of an $\Omega.3$ move is also an $\Omega.3$ move, as is the
composition of an $\Omega.3$ move with $\Omega.2$ moves. Moreover the
\textquotedblleft mirror image\textquotedblright\ of an $\Omega.3$ move --
i.e. the transformation obtained by first toggling all loops, then applying an
$\Omega.3$ move, and then toggling all loops again -- is also an $\Omega.3$
move. There are many different resulting moves, including the six from
\cite{TZ} pictured in Fig. \ref{linksbis10}.%

%TCIMACRO{\FRAME{ftbpFU}{4.7366in}{1.0473in}{0pt}{\Qcb{If every unpictured
%vertex is adjacent to either none or precisely two of the three pictured
%vertices in one of these six configurations, an $\Omega.3$ move may be
%performed by toggling all non-loop edges among the three pictured vertices.}%
%}{\Qlb{linksbis10}}{linksbis10.ps}{\special{ language "Scientific Word";
%type "GRAPHIC";  maintain-aspect-ratio TRUE;  display "USEDEF";
%valid_file "F";  width 4.7366in;  height 1.0473in;  depth 0pt;
%original-width 8.246in;  original-height 10.6969in;  cropleft "0.0971";
%croptop "0.8745";  cropright "0.9130";  cropbottom "0.7383";
%filename 'linksbis10.ps';file-properties "XNPEU";}}}%
%BeginExpansion
\begin{figure}
[ptb]
\begin{center}
\includegraphics[
trim=0.800687in 7.897522in 0.717402in 1.342461in,
height=1.0473in,
width=4.7366in
]%
{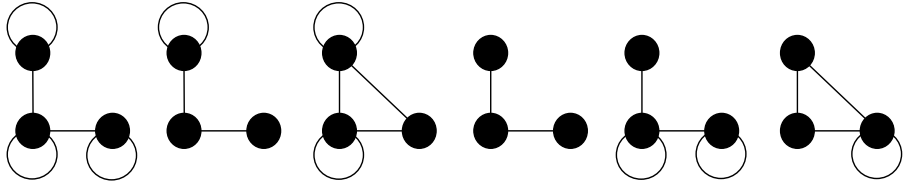}%
\caption{If every unpictured vertex is adjacent to either none or precisely
two of the three pictured vertices in one of these six configurations, an
$\Omega.3$ move may be performed by toggling all non-loop edges among the
three pictured vertices.}%
\label{linksbis10}%
\end{center}
\end{figure}
%EndExpansion

Theorem \ref{diagramcomplement} tells us how to extend the Reidemeister moves
of \cite{T1, TZ} from singly marked graphs to multiply marked graphs: simply
compose with marked local complementations.

\begin{definition}
\label{Reidemeister} A \emph{marked-graph Reidemeister move} on a marked graph
$G$ is performed by first applying marked local complementations, then
applying one of the marked-graph Reidemeister moves defined above, and then
applying marked local complementations.
\end{definition}

Two marked graphs are \textit{Reidemeister equivalent} if they can be obtained
from each other using marked local complementations and
marked-graph\ Reidemeister moves.

\begin{theorem}
\label{diagram} Let $D$ and $D^{\prime}$ be oriented link diagrams
representing the same virtual link type. Then for any Euler systems $C$ and
$C^{\prime}$ of the corresponding universe graphs, $\mathcal{L}(D^{\prime}$,
$C^{\prime})$ can be obtained from $\mathcal{L}(D$, $C)$ by using marked local
complementation and marked-graph Reidemeister moves.
\end{theorem}

Theorem \ref{diagram} follows immediately from the results of \cite{T1, TZ}
using the machinery of Section 2.

Before introducing the bracket polynomial we take a moment to discuss mirror
images. It is certainly not surprising that the mirror image of a Reidemeister
move should be considered a Reidemeister move, and separate consideration of
mirror images involved little extra work in \cite{T1, TZ}. Nevertheless it is
worth mentioning that it is not actually necessary to consider the mirror
images of $\Omega.3$ moves separately. In \cite{TZ} we adapted some
equivalences given by \"{O}stlund \cite{O} to show that the first three
$\Omega.3$ moves pictured in Fig. \ref{linksbis10} can be obtained from each
other through composition with $\Omega.2$ moves, and the second three can also
be obtained from each other. \"{O}stlund mentioned that there are two
equivalence classes of $\Omega.3$ moves, so we were content to have two
classes too. But the difference between \"{O}stlund's ascending and descending
$\Omega.3$ moves is not the same as the difference between mirror images; for
our purposes it is actually a difference that makes no difference, and it
turns out that all the $\Omega.3$ moves can be obtained directly from each
other by composition with $\Omega.2$ moves. See Fig. \ref{onethre4}, which
illustrates ways to obtain the $\Omega.3$ moves involving the third and fourth
configurations of Fig. \ref{linksbis10} from each other. (Vertices that appear
in a horizontal row in Fig. \ref{onethre4} are presumed to have the same
neighbors outside the pictured subgraph.) The sequence of moves pictured at
the top is adapted from \cite{P}.%

%TCIMACRO{\FRAME{ftbpFU}{4.3578in}{7.0837in}{0pt}{\Qcb{All of the $\Omega.3$
%moves for marked graphs are inter-related through composition with $\Omega.2$
%moves.}}{\Qlb{onethre4}}{onethre4.ps}{\special{ language "Scientific Word";
%type "GRAPHIC";  maintain-aspect-ratio TRUE;  display "USEDEF";
%valid_file "F";  width 4.3578in;  height 7.0837in;  depth 0pt;
%original-width 8.246in;  original-height 10.6969in;  cropleft "0.0856";
%croptop "0.9119";  cropright "0.7420";  cropbottom "0.0875";
%filename 'onethre4.ps';file-properties "XNPEU";}}}%
%BeginExpansion
\begin{figure}
[ptb]
\begin{center}
\includegraphics[
trim=0.705858in 0.935979in 2.127468in 0.942397in,
height=7.0837in,
width=4.3578in
]%
{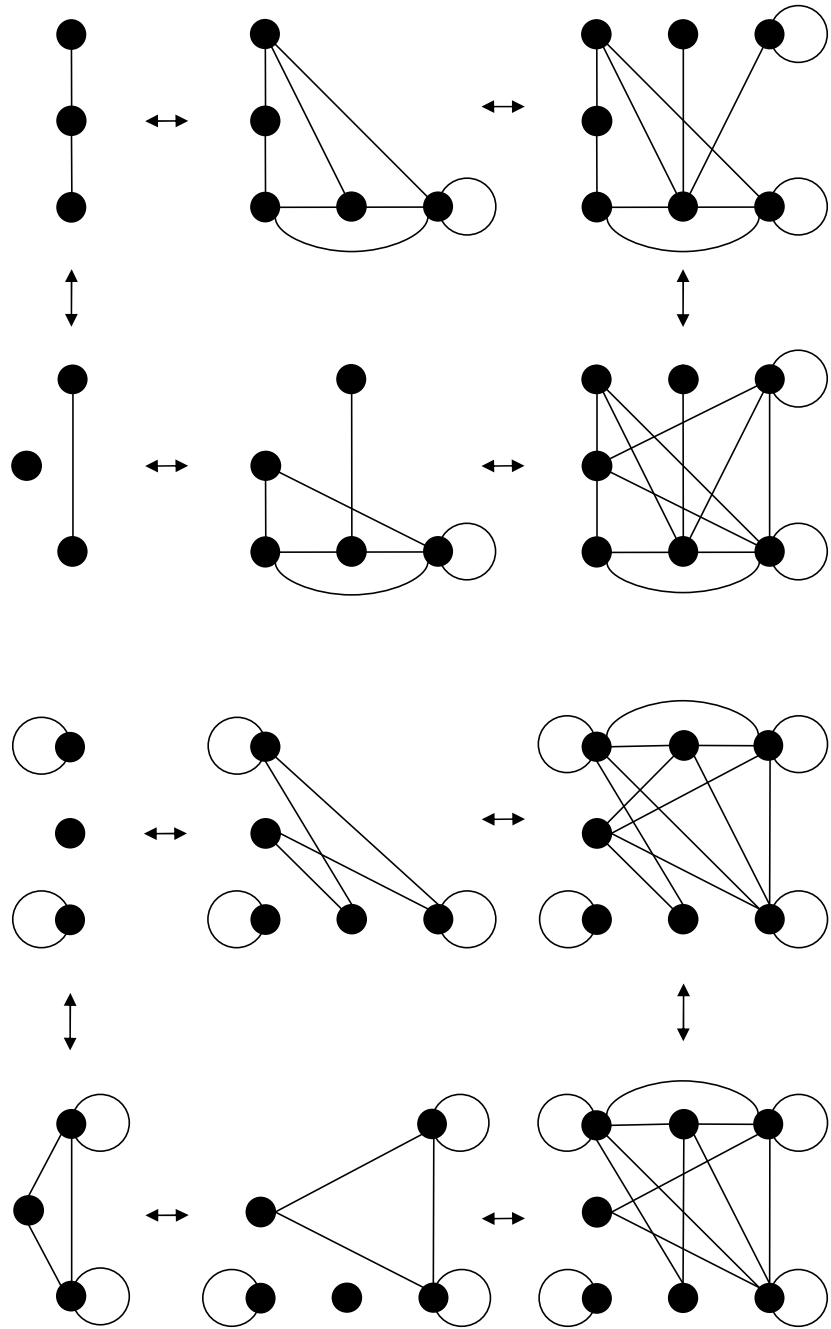}%
\caption{All of the $\Omega.3$ moves for marked graphs are inter-related
through composition with $\Omega.2$ moves.}%
\label{onethre4}%
\end{center}
\end{figure}
%EndExpansion

\section{The extended marked-graph bracket}

Suppose $U$ is any 4-regular graph, with $\phi$ free loops and $c(U)$
connected components. A \textit{circuit} in $U$ is a sequence $v_{1}$, $h_{1}
$, $h_{1}^{\prime}$, $v_{2}$, ..., $v_{k}$, $h_{k}$, $h_{k}^{\prime}$,
$v_{k+1}=v_{1}$ such that for each $i$, $h_{i}$ and $h_{i}^{\prime}$ are the
half-edges of an edge $e_{i}$ connecting $v_{i}$ to $v_{i+1}$. (It is
technically necessary to refer to half-edges because a loop is regarded as
providing two different one-edge circuits, with opposite orientations.) There
are $3^{n}$ partitions of $E(G)$ into circuits, each of which is determined by
choosing one of the three \textit{transitions} (pairings of incident
half-edges) at every vertex. Each circuit partition is also required to
include all the free loops of $U$. Let $C$ be an Euler system for $U$; choose
one of the two orientations for each circuit that appears in $U$, and let
$\vec{U}$ be the 2-in, 2-out digraph obtained from $U$ by using these
orientations to assign directions to edges. Then as indicated in Fig.
\ref{linksbis105}, the three transitions at a vertex $v$ are identified by
their relationships with $C$: one follows $C$, one is consistent with the
edge-directions of $\vec{U}$ without following $C$, and the third is
inconsistent with the edge-directions of $\vec{U}$. Note that changing the
choice of orientations for the elements of $C$ does not affect these designations.

The tool that allows us to use interlacement to describe the Kauffman bracket
is the \textit{circuit-nullity formula.} This formula has a very interesting
history; at least five different special cases have been discovered during the
last century \cite{Bu, Br, CL, Me, S, Z}. We refer to \cite{Tb} for a detailed
exposition, and only summarize the basic idea here. The three transitions
pictured in Fig. \ref{linksbis105} are represented (respectively) by three
operations on the simple interlacement graph of $\vec{U}$ with respect to $C$:
delete $v$, do nothing to $v$, and attach a loop at $v$. If $P$ is a circuit
partition of $U$ then the circuit-nullity formula states that the number of
elements of $P$ is
\[
\left\vert P\right\vert =c(U)+\nu(\mathcal{A}_{P}),
\]
where $\nu(\mathcal{A}_{P})$ is the $GF(2)$-nullity of the adjacency matrix of
the graph obtained from the interlacement graph of $\vec{U}$ with respect to
$C$ by performing, at each vertex, the operation corresponding to the
transition used in $P$.%
%TCIMACRO{\FRAME{ftbpFU}{3.9038in}{0.6253in}{0pt}{\Qcb{The three transitions at
%a vertex, in relation to an oriented Euler circuit: one transition follows the
%circuit, the second is consistent with the edge-directions determined by the
%Euler circuit, and the third is not consistent with these edge-directions.}%
%}{\Qlb{linksbis105}}{linksbis105.ps}{\special{ language "Scientific Word";
%type "GRAPHIC";  maintain-aspect-ratio TRUE;  display "USEDEF";
%valid_file "F";  width 3.9038in;  height 0.6253in;  depth 0pt;
%original-width 8.246in;  original-height 10.6969in;  cropleft "0.1297";
%croptop "0.8744";  cropright "0.7564";  cropbottom "0.8000";
%filename 'linksbis105.ps';file-properties "XNPEU";}}}%
%BeginExpansion
\begin{figure}
[ptb]
\begin{center}
\includegraphics[
trim=1.069506in 8.557521in 2.008726in 1.343530in,
height=0.6253in,
width=3.9038in
]%
{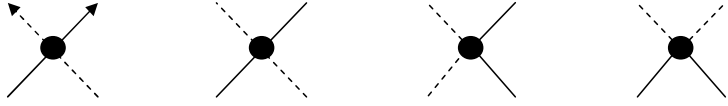}%
\caption{The three transitions at a vertex, in relation to an oriented Euler
circuit: one transition follows the circuit, the second is consistent with the
edge-directions determined by the Euler circuit, and the third is not
consistent with these edge-directions.}%
\label{linksbis105}%
\end{center}
\end{figure}
%EndExpansion

Looking at Figs. \ref{linksbis5} and \ref{linksbis105}, we see that if $U$ is
the universe of a link diagram $D$ then vertex marks determine which
transitions correspond to the $A$ and $B$ smoothings at a positive crossing as
in Table \ref{1}. The transitions corresponding to the $A$ and $B$ smoothings
at a negative crossing are simply interchanged.%

\begin{equation}%
\begin{tabular}
[c]{|ccccc|}\hline
vertex mark &  & $A$ transition &  & $B$ transition\\
&  &  &  & \\
(unmarked) &  & orientation-consistent &  & orientation-inconsistent\\
$r$ &  & orientation-inconsistent &  & orientation-consistent\\
$c$ &  & follow $C$ &  & orientation-inconsistent\\
$cr$ &  & orientation-inconsistent &  & follow $C$\\
$u$ &  & orientation-consistent &  & follow $C$\\
$ur$ &  & follow $C$ &  & orientation-consistent\\\hline
\end{tabular}
\label{1}%
\end{equation}

These considerations motivate the following definitions.

\begin{definition}
\label{adjmatrix} Let $G$ be a graph with $V(G)=\{v_{1}$, $...$, $v_{n}\}$.
The \emph{Boolean adjacency matrix} of $G$ is the $n\times n$ matrix
$\mathcal{A}(G)$ with entries in $GF(2)$ defined by: if $i\neq j$ then
$\mathcal{A}(G)_{ij}=1$ if and only if $v_{i}$ is adjacent to $v_{j}$, and
$\mathcal{A}(G)_{ii}=1$ if and only if $v_{i}$ is looped.
\end{definition}

Observe that $\mathcal{A}(G)$ is defined if $G$ has multiple edges or multiple
loops, but they do not affect it.

\begin{definition}
\label{adjT} Let $G$ be a\ marked graph with $V(G)=\{v_{1}$, $...$, $v_{n}\}$.
Suppose $T\subseteq V(G)$, and let $\Delta_{T}$ be the $n\times n$ matrix with
the following entries in $GF(2)$.%
\[
(\Delta_{T})_{ij}=\left\{
\begin{tabular}
[c]{ll}%
$0,$ & if $i\neq j$\\
$1,$ & if $i=j$ and $v_{i}\in T$ has a mark with no $r$\\
$0,$ & if $i=j$ and $v_{i}\in T$ has a mark with an $r$\\
$1,$ & if $i=j$ and $v_{i}\not \in T$ has a mark with an $r$\\
$0,$ & if $i=j$ and $v_{i}\not \in T$ has a mark with no $r$%
\end{tabular}
\right.
\]
Then $\mathcal{A}(G)_{T}$ is defined to be the submatrix of $\mathcal{A}%
(G)+\Delta_{T}$ obtained by removing the $i^{th}$ row and column if either (a)
$v_{i}$ is marked $c$ or $cr$ and $(\mathcal{A}(G)+\Delta_{T})_{ii}=0$ or (b)
$v_{i}$ is marked $u$ or $ur$ and $(\mathcal{A}(G)+\Delta_{T})_{ii}=1 $.
\end{definition}

\begin{definition}
\label{bracket} The \emph{marked-graph bracket polynomial} of a marked graph
$G$ with $\phi$ free loops and $V(G)=\{v_{1}$, $...$, $v_{n}\}$ is
\[
\lbrack G]=d^{\phi}\cdot\sum_{T\subseteq V(G)}A^{n-\left\vert T\right\vert
}B^{\left\vert T\right\vert }d^{\nu(\mathcal{A}(G)_{T})},
\]
where $\nu(\mathcal{A}(G)_{T})$ is the $GF(2)$-nullity of $\mathcal{A}(G)_{T}$.
\end{definition}

Although the definition of $[G]$ requires an ordering of $V(G)$, choosing one
ordering rather than another simply permutes the rows and columns of
$\mathcal{A}(G)$; obviously this does not affect $[G]$. The next two results
are almost as obvious.

\begin{proposition}
\label{rloop} If $G^{\prime}$ is obtained from $G$ by toggling both the loop
status and the letter $r$ in the mark of a vertex $v$, then $[G]=[G^{\prime}]$.
\end{proposition}

\begin{proof}
Table \ref{1} indicates that toggling the letter $r$ in the mark of $v$ has
the same effect on the bracket as toggling the loop status of $v$: the $A$ and
$B$ transitions at $v$ are interchanged. Consequently, toggling both the loop
status and the letter $r$ has no effect at all.
\end{proof}

\begin{theorem}
\label{brac1} If $D$ is a virtual link diagram then $[\mathcal{L}(D$, $C)]$ is
the same as the Kauffman bracket $[D]$.
\end{theorem}

\begin{proof}
For each subset $T\subseteq V(\mathcal{L}(D$, $C))$ let $S(T)$ be the Kauffman
state of $D$ that involves $B$ smoothings at the vertices of $T$ and $A$
smoothings elsewhere. The number of closed curves in $S(T)$ is related to the
binary nullity of $\mathcal{A}(G)_{T}$ by the circuit-nullity equality:
$c(S(T))=c(U)+\nu(\mathcal{A}(G)_{T})$. As $\mathcal{L}(D$, $C)$ has
$\phi=c(U)-1$ free loops, the theorem follows immediately.
\end{proof}

It follows that $[\mathcal{L}(D,C)]$ is not affected by the choice of $C$.
According to Theorem \ref{diagramcomplement}, this is equivalent to saying
that $[\mathcal{L}(D,C)]$ is invariant under marked local complementation.
This invariance actually holds for arbitrary marked graphs, not just those
that arise from link diagrams.

\begin{theorem}
\label{brac2} If $G$ is a marked graph then $[G]=[G_{cru}^{v}]$ for every
$v\in V(G)$.
\end{theorem}

Indeed, Theorem \ref{brac2} is true term by term; that is, each subset
$T\subseteq V(G)$ makes the same contribution to $[G]$ and $[G_{cru}^{v}]$.

\begin{theorem}
\label{samen} Let $G$ be a marked graph with a vertex $v$. Then for every
subset $T\subseteq V(G)$,
\[
\nu(\mathcal{A}(G)_{T})=\nu(\mathcal{A}(G_{cru}^{v})_{T}).
\]

\end{theorem}

\begin{proof}
Let $V(G)=\{v_{1},...,v_{n}\}$, with $v=v_{1}$. Let $\Delta_{T}=\Delta_{T}(G)$
and $\Delta_{T}^{\prime}=\Delta_{T}(G_{cru}^{v})$ be the diagonal matrices
that appear in Definition \ref{adjT}; they differ in the diagonal entries
corresponding to neighbors of $v$, and also in the diagonal entry
corresponding to $v$ if the mark of $v$ is $c$ or $cr$.

Suppose $i\geq2$. We claim that Definition \ref{adjT} tells us to remove the
$i^{th}$ row and column of $\mathcal{A}(G)+\Delta_{T}$ in constructing
$\mathcal{A}(G)_{T}$ if and only if it tells us to remove the $i^{th}$ row and
column of $\mathcal{A}(G_{cru}^{v})+\Delta_{T}^{\prime}$ in constructing
$\mathcal{A}(G_{cru}^{v})_{T}$. If $v_{i}$ is not a neighbor of $v$, then
$v_{i}$ has the same mark in $G_{cru}^{v}$ as in $G$, and $(\mathcal{A}%
(G)+\Delta_{T})_{ii}=(\mathcal{A}(G_{cru}^{v})+\Delta_{T}^{\prime})_{ii}$, so
the claim is satisfied. If $v_{i}$ is a neighbor of $v$ marked $c$ or $u$ in
$G$, then $v_{i}$ is marked $ur$ or $cr$ (respectively) in $G_{cru}^{v}$;
moreover, $(\mathcal{A}(G)+\Delta_{T})_{ii}\neq(\mathcal{A}(G_{cru}%
^{v})+\Delta_{T}^{\prime})_{ii}$ because $(\Delta_{T})_{ii}\neq(\Delta
_{T}^{\prime})_{ii}$. Consequently the claim is satisfied. Similarly, the
claim is satisfied if $v_{i}$ is a neighbor of $v$ marked $cr$ or $ur$.
Finally, if $v_{i}$ is a\ neighbor of $v $ that is unmarked or marked $r$ in
$G$ then $v_{i}$ is marked $r$ or unmarked (respectively) in $G_{cru}^{v}$;
either way Definition \ref{adjT} does not tell us to remove the $i^{th}$ row
and column of $\mathcal{A}(G)+\Delta_{T}$ or $\mathcal{A}(G_{cru}^{v}%
)+\Delta_{T}^{\prime}$. This completes the proof of the claim.

If $(\mathcal{A}(G)+\Delta_{T})_{11}=0$ and $v$ is not marked $c$ or $cr$ in
$G$, then the same is true in $G_{cru}^{v}$ and we verify that
\[
\nu(\mathcal{A}(G)_{T})=\nu%
\begin{pmatrix}
0 & \mathbf{1} & \mathbf{0}\\
\mathbf{1} & M_{11} & M_{12}\\
\mathbf{0} & M_{21} & M_{22}%
\end{pmatrix}
=\nu%
\begin{pmatrix}
0 & \mathbf{1} & \mathbf{0}\\
\mathbf{1} & \bar{M}_{11} & M_{12}\\
\mathbf{0} & M_{21} & M_{22}%
\end{pmatrix}
=\nu(\mathcal{A}(G_{cru}^{v})_{T})
\]
by adding the top row to every row in the second set of rows. (Bold numerals
denote rows and columns whose entries are all the same, and $\bar{M}_{11}$
denotes the matrix obtained by toggling every entry of $M_{11}$.)

If $(\mathcal{A}(G)+\Delta_{T})_{11}=0$ and $v$ is marked $c$ or $cr$ in $G$,
then $(\mathcal{A}(G_{cru}^{v})+\Delta_{T}^{\prime})_{11}=1$ and $v$ is marked
$cr$ or $c$ (respectively) in $G_{cru}^{v}$, so%
\[
\nu%
\begin{pmatrix}
M_{11} & M_{12}\\
M_{21} & M_{22}%
\end{pmatrix}
=\nu%
\begin{pmatrix}
1 & \mathbf{1} & \mathbf{0}\\
\mathbf{0} & M_{11} & M_{12}\\
\mathbf{0} & M_{21} & M_{22}%
\end{pmatrix}
=\nu%
\begin{pmatrix}
1 & \mathbf{1} & \mathbf{0}\\
\mathbf{1} & \bar{M}_{11} & M_{12}\\
\mathbf{0} & M_{21} & M_{22}%
\end{pmatrix}
\]
verifies that $\nu(\mathcal{A}(G)_{T})=\nu(\mathcal{A}(G_{cru}^{v})_{T})$. The
same calculation applies if $(\mathcal{A}(G)+\Delta_{T})_{11}=1$ and $v$ is
marked $u$ or $ur$ in $G$.

If $(\mathcal{A}(G)+\Delta_{T})_{11}=1$ and $v$ is not marked $u$ or $ur$ in
$G$, a similar calculation shows that $\nu(\mathcal{A}(G)_{T})=\nu
(\mathcal{A}(G_{cru}^{v})_{T})$.%
\[
\nu%
\begin{pmatrix}
1 & \mathbf{1} & \mathbf{0}\\
\mathbf{1} & M_{11} & M_{12}\\
\mathbf{0} & M_{21} & M_{22}%
\end{pmatrix}
=\nu%
\begin{pmatrix}
1 & \mathbf{1} & \mathbf{0}\\
\mathbf{0} & \bar{M}_{11} & M_{12}\\
\mathbf{0} & M_{21} & M_{22}%
\end{pmatrix}
=\nu%
\begin{pmatrix}
\bar{M}_{11} & M_{12}\\
M_{21} & M_{22}%
\end{pmatrix}
.
\]

\end{proof}

With Theorems \ref{brac1} and \ref{brac2} in hand, we conclude that Kauffman's
classical construction of the Jones polynomial from the bracket \cite{Kau}
extends directly to multiply marked graphs.

\begin{definition}
The \emph{reduced marked-graph bracket polynomial} $\left\langle
G\right\rangle $ is the image of the three-variable marked-graph bracket under
the evaluations $B\mapsto A^{-1}$ and $d\mapsto-A^{-2}-A^{2}$.
\end{definition}

\begin{definition}
\label{Jones} The\emph{\ marked-graph Jones polynomial} of a graph with $\ell$
looped vertices and $n-\ell$ unlooped vertices is%
\[
V_{G}(t)=(-1)^{n}\cdot t^{(3n-6\ell)/4}\cdot\left\langle G\right\rangle
(t^{-1/4}).
\]

\end{definition}

\begin{theorem}
The reduced bracket is invariant under marked local complementations and
marked-graph\ Reidemeister moves of types $\Omega.2$ and $\Omega.3$. The
marked-graph Jones polynomial is invariant under marked local complementations
and all three types of marked-graph\ Reidemeister moves.
\end{theorem}

\begin{proof}
The invariance of $\left\langle G\right\rangle $ under marked local
complementations follows from Theorem \ref{brac2}, and the invariance of
$\left\langle G\right\rangle $ under $\Omega.2$ and $\Omega.3$ moves is proven
using the same matrix-nullity arguments that appear in \cite{T1}. The
invariance of the Jones polynomial follows from the fact that the effects of
an $\Omega.1$ move on $(-1)^{n}\cdot t^{(3n-6\ell)/4}$ and $\left\langle
G\right\rangle (t^{-1/4})$ cancel each other.
\end{proof}

\section{Some equivalence relations}

In this section we briefly discuss several equivalence relations that come to
mind when we consider links and marked graphs.

1. The finest interesting equivalence relation on marked graphs is generated
by the marked pivots of Definition \ref{mpivot}. As proved in \cite{T1},
results of Kotzig \cite{K}, Pevzner \cite{Pev} and Ukkonen \cite{U} imply that
if $D$ is an oriented link diagram and $C$ is any directed Euler system of $D$
then the equivalence class of $\mathcal{L}(D,C)$ is the set of interlacement
graphs $\mathcal{L}(D,C^{\prime})$ corresponding to various directed Euler
systems $C^{\prime}$ of $D$.

2. A strictly coarser equivalence relation on marked graphs is generated by
marked local complementation. Theorem \ref{diagramcomplement} tells us that
this relation extends the relation tying $\mathcal{L}(D$, $C)$ to
$\mathcal{L}(D$, $C^{\prime})$ for arbitrary Euler systems $C$ and $C^{\prime
}$ in an oriented link diagram $D$, i.e. if $\mathcal{G}_{cru}$ denotes the
set of equivalence classes of multiply marked graphs under this relation and
$\mathcal{D}$ denotes the set of oriented link diagrams then the marked
interlacement graph construction provides a well-defined function
$\mathcal{L}:\mathcal{D}\rightarrow\mathcal{G}_{cru}$. Two singly marked
graphs that arise from link diagrams are equivalent under this relation if and
only if they are equivalent under the first relation; we do not know whether
or not this property extends to arbitrary singly marked graphs. Theorem
\ref{brac2} tells us that the bracket $[~]$ is well-defined on $\mathcal{G}%
_{cru}$, and Theorem \ref{brac1} tells us that the Kauffman bracket is defined
on $\mathcal{D}$ by the composition $[~]\circ\mathcal{L}$.

3. \textit{Reidemeister equivalence} is the equivalence relation $\sim$ on
marked graphs generated by marked-graph\ Reidemeister moves and marked local
complementation. Theorem \ref{diagram} tells us that this relation extends the
relation tying $\mathcal{L}(D$, $C)$ to $\mathcal{L}(D^{\prime}$, $C^{\prime
})$ for arbitrary Euler systems $C$ and $C^{\prime}$ in diagrams $D$ and
$D^{\prime}$ representing the same virtual link type. That is, if $\sim$
denotes the link type equivalence relation on $\mathcal{D}$ then
$\mathcal{L}:\mathcal{D}\rightarrow\mathcal{G}_{cru}$ induces a well-defined
function ${\widetilde{\mathcal{L}}}:\mathcal{D}/\!\sim~\rightarrow
~\mathcal{G}_{cru}/\!\sim$. The marked-graph Jones polynomial is a
well-defined function on $\mathcal{G}_{cru}/\!\sim$, whose composition with
${\widetilde{\mathcal{L}}}$ is the familiar Jones polynomial of virtual links.

4. \textit{Regular isotopy} is the finer equivalence relation that does not
involve $\Omega.1$ moves. The reduced bracket and the writhe are well-defined
modulo regular isotopy.

5. There are several equivalence relations on link diagrams that are connected
to the functions $\mathcal{L}$ and ${\widetilde{\mathcal{L}}}$. For instance,
suppose $D$ and $D^{\prime}$ are link diagrams and there is an isomorphism
between the oriented universe graphs $\vec{U}$ and $\overrightarrow{U^{\prime
}}$ that maps the directed circuits corresponding to the link components in
$D$ to the directed circuits corresponding to the link components in
$D^{\prime}$. Then $\mathcal{L}$ cannot distinguish between $D$ and
$D^{\prime}$. If $D_{1}$ and $D_{2}$ are diagrams then $\mathcal{L}$ cannot
distinguish between different connected sums $D_{1}\#D_{2}$, or between a
split union $D_{1}\cup D_{2}$ and a diagram obtained by adding a free loop to
a connected sum $D_{1}\#D_{2}$.

6. Every marked graph $G$ has an $r$\textit{-simplification} $G_{r}$ obtained
by removing the $r$ from every vertex of $G$ whose mark includes one, and
toggling the loop status of each such vertex. Proposition \ref{rloop} tells us
that $[G]=[G_{r}]$, so when discussing $[G]$ it is reasonable to consider the
equivalence relation generated by local complementation and $r$%
-simplification, and when discussing $\left\langle G\right\rangle $ it is
reasonable to consider the equivalence relation generated by regular isotopy
and $r$-simplification.

The first versions of this paper incorporated $r$-simplification throughout.
However, in link diagrams the loop status and the $r$ status of a crossing
reflect different kinds of information: the loop status reflects the sign of
the crossing, and the $r$ status reflects the way an Euler circuit traverses
the crossing. Consequently $r$-simplification involves the loss of valuable
information about link diagrams. For example, Fig. \ref{linksbis104} shows
that even though $\mathcal{L}(D,C)$ determines both the writhe and the Jones
polynomial, $\mathcal{L}(D,C)_{r}$ determines neither. (In any diagram of a
multi-component link, reversing the orientation of one link component will
have the same effect: every crossing involving that link component and another
will have both its loop status and its $r$ status toggled.)%

%TCIMACRO{\FRAME{ftbpFU}{4.6086in}{2.3298in}{0pt}{\Qcb{Diagrams of the positive
%and negative Hopf link have associated interlacement graphs whose
%$r$-simplifications are isomorphic.}}{\Qlb{linksbis104}}{linksbis104.ps}%
%width 4.6086in;  height 2.3298in;  depth 0pt;  original-width 8.246in;
%original-height 10.6969in;  cropleft "0.1293";  croptop "0.9124";
%cropright "0.8701";  cropbottom "0.6255";
%filename 'linksbis104.ps';file-properties "XNPEU";}}}%
%BeginExpansion
\begin{figure}
[ptb]
\begin{center}
\includegraphics[
trim=1.066208in 6.690911in 1.071155in 0.937048in,
height=2.3298in,
width=4.6086in
]%
{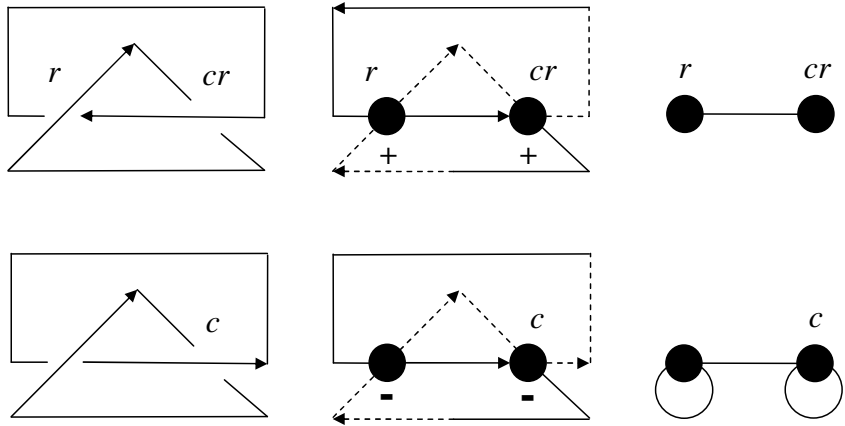}%
\caption{Diagrams of the positive and negative Hopf link have associated
interlacement graphs whose $r$-simplifications are isomorphic.}%
\label{linksbis104}%
\end{center}
\end{figure}
%EndExpansion

\section{Graph-links}

\begin{definition}
\cite{IM1} A \emph{labeled graph} $G$ is a simple graph each of whose vertices
is labeled by a pair $(a$, $\alpha)\in\{0$, $1\}\times\{+$, $-\}$.
\end{definition}

\begin{definition}
\cite{IM1} A \emph{graph-link} is an equivalence class of labeled graphs under
the equivalence relation generated by the following operations.

$\Omega_{g}1$. Adjoin or remove an isolated vertex with label $(0$, $\pm)$.

$\Omega_{g}2$. Adjoin or remove a pair of non-adjacent (resp. adjacent)
vertices that are labeled $(0$, $\pm\alpha)$ (resp. $(1$, $\pm\alpha)$) and
have the same adjacencies with other vertices.

$\Omega_{g}3$. Suppose $G$ has three distinct vertices $v$, $w$, $x$ labeled
$(0$, $-)$, such that the only neighbors of $x$ are $v$ and $w$, which are not
neighbors of each other. Then change the labels of $v$ and $w$ to $(0$, $+)$,
make $x$ adjacent to every vertex that is adjacent to precisely one of $v$,
$w$, and remove the edges connecting $x$ to $v$ and $w$. (The inverse of this
operation is also an $\Omega_{g}3$ move.)

$\Omega_{g}4$. Suppose $G$ has two adjacent vertices $v$ and $w$ labeled $(0
$, $\alpha)$ and $(0$, $\beta)$. Replace $G$ with $((G^{v})^{w})^{v}$ and then
change the labels of $v$ and $w$ to $(0$, $-\beta)$ and $(0$, $-\alpha)$
respectively. (The inverse is also an $\Omega_{g}4$ move.)

$\Omega_{g}4^{\prime}$. Suppose $G$ has a vertex $v$ with label $(1$,
$\alpha)$. Replace $G$ with $G^{v}$, change the label of $v$ to $(1$,
$-\alpha)$, and change the label of each neighbor of $v$ by changing the first
coordinate and leaving the second coordinate the same. (The inverse is also an
$\Omega_{g}4^{\prime}$ move.)
\end{definition}

\begin{definition}
The marked graph $mark(G)$ associated to a labeled graph $G$ is obtained by
preserving all non-loop edges and changing labels to loop-mark combinations as
follows: $(0$, $+)$ becomes $u$ with a loop; $(0$, $-)$ becomes $u$ with no
loop; $(1$, $+)$ becomes $c$ with no loop; and $(1$, $-)$ becomes $c$ with a loop.
\end{definition}

The fact that the vertex marks in $mark(G)$ do not involve the letter $r$
indicates that the relationship between marked graphs and graph-links involves
the notion of $r$-simplification mentioned in Section 5.%

%TCIMACRO{\FRAME{ftbpFU}{4.772in}{5.7363in}{0pt}{\Qcb{The marked-graph version
%of the graph-link $\Omega_{g}3$ move is a composition of $r$-simplifications
%with a marked local complementation, a marked pivot, one of the $\Omega.3$
%moves of Figure \ref{linksbis10}, and three marked local complementations.}%
%}{\Qlb{linksbis11}}{linksbis11.ps}{\special{ language "Scientific Word";
%type "GRAPHIC";  maintain-aspect-ratio TRUE;  display "USEDEF";
%valid_file "F";  width 4.772in;  height 5.7363in;  depth 0pt;
%original-width 8.246in;  original-height 10.6969in;  cropleft "0.0809";
%croptop "0.8874";  cropright "0.9026";  cropbottom "0.1250";
%filename 'linksbis11.ps';file-properties "XNPEU";}}}%
%BeginExpansion
\begin{figure}
[ptb]
\begin{center}
\includegraphics[
trim=0.667101in 1.337113in 0.803160in 1.204471in,
height=5.7363in,
width=4.772in
]%
{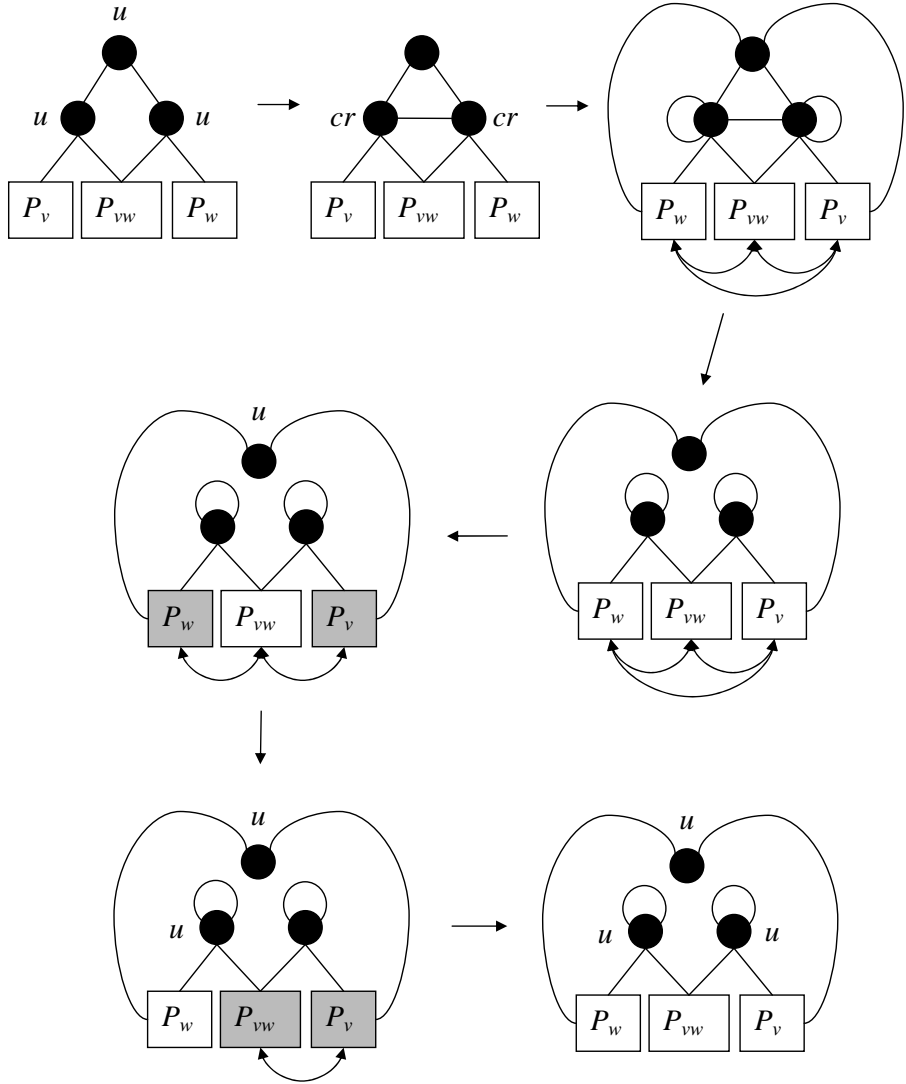}%
\caption{The marked-graph version of the graph-link $\Omega_{g}3$ move is a
composition of $r$-simplifications with a marked local complementation, a
marked pivot, one of the $\Omega.3$ moves of Figure \ref{linksbis10}, and
three marked local complementations.}%
\label{linksbis11}%
\end{center}
\end{figure}
%EndExpansion

\begin{theorem}
\label{graphlink} If two labeled graphs $G$ and $H$ define the same graph-link
then the corresponding marked graphs $mark(G)$ and $mark(H)$ are equivalent
under marked local complementation, $r$-simplification and marked-graph
Reidemeister moves.
\end{theorem}

\begin{proof}
An $\Omega_{g}4^{\prime}$ move corresponds to the $r$-simplification of a
marked local complementation at $v$, and an $\Omega_{g}4$ move corresponds to
the $r$-simplification of a marked pivot with respect to $v$ and $w$.

The first three types of graph-link operations correspond to
marked-graph\ Reidemeister moves. An $\Omega_{g}1$ move corresponds to a
marked-graph $\Omega.1$ move. An $\Omega_{g}2$ move performed on vertices
labeled $(0$, $\pm\alpha)$ corresponds to an instance of Definition \ref{R2}
(b) in the $r $-simplification of $(G_{cru}^{v})_{cru}^{w}$. An $\Omega_{g}2$
move performed on vertices labeled $(1$, $\pm\alpha)$, on the other hand,
corresponds to an instance of Definition \ref{R2} (a) in the $r$%
-simplification of $G_{cru}^{vw}$.

The $\Omega_{g}3$ moves are more complicated. Let $P_{v}$, $P_{w}$,
$P_{vw}\subseteq V(G)-\{v$, $w$, $x\}$ consist of those vertices that are
adjacent to $v$ and not $w$, $w$ and not $v$, and both $v$ and $w$
(respectively). Also, let $H$ be the labeled graph that results from the
$\Omega_{g}3$ move. Let $G^{\prime}$ be the marked graph obtained from
$mark(G)$ by performing a marked local complementation at $x$, a marked pivot
with respect to $v$ and $w$, and then an $r$-simplification. Then $v$, $w$ and
$x$ are all unmarked in $G^{\prime}$, and they induce a subgraph isomorphic to
the third one pictured in Fig. \ref{linksbis10}, with $x$ unlooped. According
to Lemma \ref{pivot}, the neighbors of $v$ in $G^{\prime}$ are the elements of
$P_{vw}\cup P_{w}$, the neighbors of $w$ in $G^{\prime}$ are the elements of
$P_{v}\cup P_{vw}$, and the neighbors of $x$ in $G^{\prime}$ are the elements
of $P_{v}\cup P_{w}$. An $\Omega.3$ move performed on this subgraph of
$G^{\prime}$ results in a graph $G^{\prime\prime}$ that resembles $mark(H)$ in
that no two of $v$, $w$, $x$ are neighbors, $v$ and $w$ are looped, and $x$ is
not looped. However $v$, $w$ and $x$ are all unmarked, while in $mark(H)$ they
are all marked $u$; also the adjacencies among their neighbors do not match
those of $mark(H)$, because of the toggling of adjacencies between vertices
from different elements of $\{P_{v}$, $P_{w}$, $P_{vw}\}$. Both problems are
solved by applying $r$-simplifications and marked local complementations at
$x$, $v$ and $w$. (The marked pivot in the second step exchanges the neighbors
of $v$ and $w$, but of course this is insignificant up to isomorphism.)

This process is pictured in Fig. \ref{linksbis11}. At the top left of the
figure we see $mark(G)$; the vertices $v$, $w$, $x$ are not named in the
figure but they are determined by their neighborhoods. Moving from left to
right along the top row we see the result of applying a marked local
complementation at $x$, \ and then a marked pivot with respect to $v$ and $w$
followed by an $r$-simplification. (The double-headed arrows indicate the
toggling of adjacencies between vertices in different elements of $\{P_{v}$,
$P_{w}$, $P_{vw}\}$.) After applying an $\Omega.3$ move we obtain
$G^{\prime\prime}$, the graph pictured on the right in the second row of the
figure. The graph pictured to the left of $G^{\prime\prime}$ is $((G^{\prime
\prime})_{cru}^{x})_{r}$; the gray boxes indicate the toggling of loops and
non-loop edges within $P_{w}$ and $P_{v}$. The last two graphs are obtained by
marked local complementations first at $v$ and then at $w$, followed by $r$-simplifications.
\end{proof}

Theorem \ref{graphlink} does not completely describe the relationship between
graph-links and marked graphs. On the one hand, some of the marked-graph local
complementation and Reidemeister moves do not occur among the
graph-link\ Reidemeister moves. For instance there is no need for $\Omega
_{g}4^{\prime}$ moves at vertices with labels $(0,\alpha)$, because the
corresponding $\kappa$-transformations would not produce rotating circuits.
This difference may allow some labeled graphs $G$ and $H$ to define
inequivalent graph-links even if $mark(G)$ and $mark(H)$ are equivalent under
Reidemeister moves and $r$-simplification. On the other hand, the fact that
equivalence of graph-links is associated with $r$-simplification raises the
possibility that there may be labeled graphs $G$ and $H$ that define the same
graph-link, but whose associated marked graphs are not Reidemeister equivalent.

\section{Recursion}

We begin developing the recursion of Theorem \ref{recursion} by discussing the
relationships among the bracket polynomials of graphs that differ only in the
loop-mark combination at a single vertex $v$. Denote a graph obtained from $G$
by changing only the loop-mark combination at $v$ by $G(v,x)$, where $x$ tells
us how $v$ has been changed: in $G(v,c,\ell)$ the vertex $v$ is marked $c$ and
looped, in $G(v,ur)$ the vertex $v$ is marked $ur$ and unlooped, in
$G(v,\ell)$ the vertex $v$ is unmarked and looped, in $G(v)$ the vertex $v$ is
unmarked and unlooped, and so on. Observe that the notations $G-v$ and
$G_{cru}^{v}-v$ are unambiguous, because these graphs are not affected if we
change the loop-mark status of $v$.

Split $[G(v)]$ into two separate sums as follows:%
\begin{align*}
S  & =d^{\phi}\cdot\sum_{v\not \in T\subseteq V(G)}A^{n-\left\vert
T\right\vert }B^{\left\vert T\right\vert }d^{\nu(\mathcal{A}(G(v))_{T})}\text{
}\\
& ~\\
\text{and }S^{\prime}  & =d^{\phi}\cdot\sum_{v\in T\subseteq V(G)}%
A^{n-\left\vert T\right\vert }B^{\left\vert T\right\vert }d^{\nu
(\mathcal{A}(G(v))_{T})}.
\end{align*}
We claim that $S^{\prime}=B[G_{cru}^{v}-v]$. If $v\in T\subseteq
V(G(v))=V(G(v)_{cru}^{v})$ then Theorem \ref{samen} tells us that
$\nu(\mathcal{A}(G(v))_{T})=\nu(\mathcal{A}(G(v)_{cru}^{v})_{T})$, and
Definition \ref{mlc} tells us that as $v$ is unlooped and unmarked in $G(v)$,
it is unlooped and marked $u$ in $G(v)_{cru}^{v}$. Definition \ref{adjT} then
tells us that the row and column of $\mathcal{A}(G(v)_{cru}^{v})+\Delta_{T}$
corresponding to $v$ are deleted in obtaining $\mathcal{A}(G(v)_{cru}^{v}%
)_{T}$; it follows that $\mathcal{A}(G(v)_{cru}^{v})_{T}=\mathcal{A}%
(G_{cru}^{v}-v)_{T-\{v\}}$, and consequently $\nu(\mathcal{A}(G_{cru}%
^{v}-v)_{T-\{v\}})=\nu(\mathcal{A}(G(v)_{cru}^{v})_{T})=\nu(\mathcal{A}%
(G(v))_{T})$. Summing over $T$ yields $B[G_{cru}^{v}-v]=S^{\prime}$.

Definitions \ref{adjT} and \ref{bracket} tell us that changing the mark of $v
$ to $r$ has the same effect as toggling between $v\in T$ and $v\notin T$;
hence $[G(v,r)]=A^{-1}BS+B^{-1}AS^{\prime}=A^{-1}BS+A[G_{cru}^{v}-v]$.

Now split $[G(v,u)]$ into two separate sums as follows:%
\begin{align*}
S  & =d^{\phi}\cdot\sum_{v\not \in T\subseteq V(G)}A^{n-\left\vert
T\right\vert }B^{\left\vert T\right\vert }d^{\nu(\mathcal{A}(G(v,u))_{T}%
)}\text{ }\\
& ~\\
\text{and }S^{\prime\prime}  & =d^{\phi}\cdot\sum_{v\in T\subseteq
V(G)}A^{n-\left\vert T\right\vert }B^{\left\vert T\right\vert }d^{\nu
(\mathcal{A}(G(v,u))_{T})}.
\end{align*}
Definition \ref{adjT} tells us that $\mathcal{A}(G(v))_{T}=\mathcal{A}%
(G(v,u))_{T}$ if $v\notin T$, so this $S$ is the same as the sum denoted $S$
in the above analysis of $[G(v)]$. Definition \ref{adjT} also tells us that
$\mathcal{A}(G(v,u))_{T}=\mathcal{A}(G-v)_{T-\{v\}}$ if $v\in T$; summing over
these $T$ we see that $S^{\prime\prime}=B[G-v]$.

Definitions \ref{adjT} and \ref{bracket} imply that $[G(v,ur)]=A^{-1}%
BS+B^{-1}AS^{\prime\prime}=A^{-1}BS+A[G-v]$.

Finally, observe that Definition \ref{adjT} implies $\mathcal{A}%
(G(v,c))_{T}=\mathcal{A}(G(v))_{T}$ if $v\in T$ and $\mathcal{A}%
(G(v,c))_{T}=\mathcal{A}(G(v,u))_{T\cup\{v\}}$ if $v\not \in T$. Consequently
$[G(v,c)]=S^{\prime}+AB^{-1}S^{\prime\prime}=$ $B[G_{cru}^{v}-v]+A[G-v]$.
Changing the mark of $v$ to $cr$ reverses the coefficients:
$[G(v,cr)]=A[G_{cru}^{v}-v]+B[G-v]$.

In sum, we have the following equalities:%
\begin{align}
\lbrack G(v)]  & =S+B[G_{cru}^{v}-v]\tag{7.1}\label{equalities}\\
\lbrack G(v,r)]  & =A^{-1}BS+A[G_{cru}^{v}-v]\nonumber\\
\lbrack G(v,c)]  & =A[G-v]+B[G_{cru}^{v}-v]\nonumber\\
\lbrack G(v,cr)]  & =B[G-v]+A[G_{cru}^{v}-v]\nonumber\\
\lbrack G(v,u)]  & =S+B[G-v]\nonumber\\
\lbrack G(v,ur)]  & =A^{-1}BS+A[G-v]\nonumber
\end{align}
Recall that Proposition \ref{rloop} tells us that $[G(v,\ell)]=[G(v,r)]$,
$[G(v,ur,\ell)]=[G(v,u)]$ and so on. Note also that in case $D$ is a link
diagram with $G=\mathcal{L}(D,C)$, the six equalities of (\ref{equalities})
correspond to the six different ways the Euler system $C$ might be related to
the $A$ and $B$ smoothings at the crossing of $D$ corresponding to $v$; see
Fig. \ref{linksbis5}.

The proof of Theorem \ref{recursion} is now very simple. Parts (a) -- (c)
follow from Definition \ref{bracket}, parts (e) and (f) follow from Theorem
\ref{brac2}, and part (d) follows from Proposition \ref{rloop} and the
formulas for $[G(v,c)]$ and $[G(v,cr)]$ in (\ref{equalities}). To perform a
computation, first use (a) -- (d) to remove free loops, isolated vertices and
vertices marked $c$ or $cr$. Suppose a nonempty graph has no vertex marked $c$
or $cr$, and no isolated vertex. If $v$ has a neighbor marked $u$ or $ur$ then
in $G_{cru}^{v}$ every such neighbor is marked $c$ or $cr$, and can be removed
with (d). If there is no vertex with a neighbor marked $u$ or $ur$ then there
are two neighbors $v$ and $w$ each of which is either unmarked or marked $r$;
$w$ is marked $u$ or $ur$ in $G_{cru}^{w}$, so (e) may be applied to $v$ in
$G_{cru}^{w}$.

If we compare Theorem \ref{recursion} to the recursions discussed in \cite{T1,
T3, TZ}, we see that two recursive steps have been removed in favor of marked
local complementations. The marked local complementations are preferable
because they do not involve replacing one graph with two or three graphs, but
the old recursive steps are still valid.

\begin{proposition}
\label{oldrec}If $v\in V(G)$ is looped and unmarked then%
\[
\lbrack G]=A^{-1}B[G-\{v,v\}]+(A-A^{-1}B^{2})[G_{cru}^{v}-v]\text{, }%
\]
where $G-\{v,v\}$ is obtained from $G$ by removing the loop at $v$. Also, if
$v$ and $w$ are two unlooped, unmarked neighbors in $G$ then
\[
\lbrack G]=A^{2}[G_{cru}^{vw}-v-w]+AB[(G_{cru}^{w})_{cru}^{v}-v-w]+B[G_{cru}%
^{v}-v].
\]

\end{proposition}

\begin{proof}
The first part follows immediately from the formulas (\ref{equalities}).

Suppose $v\neq w\in V(G)$ are adjacent, unlooped, and unmarked. Let
\begin{gather*}
S_{00}=\sum_{\substack{T\subseteq V(G) \\v\notin T,w\notin T}}A^{n-\left\vert
T\right\vert }B^{\left\vert T\right\vert }d^{\nu(\mathcal{A}(G)_{T})}\text{,
}S_{01}=\sum_{\substack{T\subseteq V(G) \\v\notin T,w\in T }}A^{n-\left\vert
T\right\vert }B^{\left\vert T\right\vert }d^{\nu(\mathcal{A}(G)_{T})}\\
~\\
\text{and }S_{1}=\sum_{\substack{T\subseteq V(G) \\v\in T}}A^{n-\left\vert
T\right\vert }B^{\left\vert T\right\vert }d^{\nu(\mathcal{A}(G)_{T})}\text{ .}%
\end{gather*}
As discussed in the second paragraph of this section, $S_{1}=B[G_{cru}^{v}-v]
$.

Suppose $v,w\notin$ $T\subseteq V(G)$. Theorem \ref{samen} tells us
$\nu(\mathcal{A}(G)_{T})=\nu(\mathcal{A}(G_{cru}^{vw})_{T})$. As $v$ and $w$
are both unmarked in $G$, Lemma \ref{pivot} and Definition \ref{mpivot} tell
us that $v$ and $w$ are both marked $c$ in $G_{cru}^{vw}$. As they are both
unlooped and not in $T$, the definition of $\mathcal{A}(G_{cru}^{vw})_{T}$
involves deleting the rows and columns of $\mathcal{A}(G_{cru}^{vw}%
)+\Delta_{T}$ corresponding to both $v$ and $w$. Consequently $\mathcal{A}%
(G_{cru}^{vw})_{T}=\mathcal{A}(G_{cru}^{vw}-v-w)_{T}$. Summing over $T$, we
see that
\[
S_{00}=A^{2}[G_{cru}^{vw}-v-w].
\]

Suppose $v\notin T\subseteq V(G)$ and $w\in T$. Theorem \ref{samen} tells us
that $\nu(\mathcal{A}(G)_{T})=\nu(\mathcal{A}((G_{cru}^{w})_{cru}^{v})_{T}) $.
In $G$, $v$ and $w$ are both unlooped and unmarked; in $G_{cru}^{w}$, $w$ is
unlooped and marked $u$ and $v$ is unlooped and marked $r$; and in
$(G_{cru}^{w})_{cru}^{v}$, $v$ is unlooped and marked $ur$ and $w$ is unlooped
and marked $cr$. As $v\notin T$ and $w\in T$, the definition of $\mathcal{A}%
((G_{cru}^{w})_{cru}^{v})_{T}$ involves deleting the rows and columns of
$\mathcal{A}((G_{cru}^{w})_{cru}^{v})+\Delta_{T}$ corresponding to $v$ and
$w$; summing over $T$ yields%
\[
S_{01}=AB[(G_{cru}^{w})_{cru}^{v}-v-w].
\]

$[G]=S_{1}+S_{00}+S_{01}$, so the second equality of the proposition follows.
\end{proof}

As noted in \cite{T1} and \cite{TZ}, for link diagrams the formulas of
Proposition \ref{oldrec} correspond to well-known properties of the Kauffman
bracket. The first corresponds to the Kauffman bracket's switching formula
(denoted $A\chi-A^{-1}\bar{\chi}=(A^{2}-A^{-2})\asymp$ in \cite{Kd}), and the
second corresponds to a double use of the basic recursion of the Kauffman
bracket, $[D]=A^{2}[D_{AA}]+AB[D_{AB}]+B[D_{B}]$. The hypothesis
\textquotedblleft no neighbor of $v$ is marked\textquotedblright\ appeared
when the formulas of Proposition \ref{oldrec} were used in \cite{T1} and
\cite{T3}, but this hypothesis was only necessary because we used ordinary
(unmarked) local complementation there. The effect of the hypothesis is to
restrict attention to situations in which the unmarked local complements are
the same as $r$-simplifications of marked local complements.

\section{Vertex weights}

In \cite{T3} we discuss several advantages of extending the marked-graph
bracket to graphs given with \textit{vertex weights}, i.e. functions $\alpha$
and $\beta$ mapping $V(G)$ into some commutative ring $R$. The weighted form
of the bracket is defined by using the weights in place of $A$ and $B$ in
Definition \ref{bracket}:%
\[
\lbrack G]=d^{\phi}\cdot\sum_{T\subseteq V(G)}(%
%TCIMACRO{\dprod \limits_{v\notin T}}%
%BeginExpansion
{\displaystyle\prod\limits_{v\notin T}}
%EndExpansion
\alpha(v))(%
%TCIMACRO{\dprod \limits_{t\in T}}%
%BeginExpansion
{\displaystyle\prod\limits_{t\in T}}
%EndExpansion
\beta(t)))d^{\nu(\mathcal{A}(G)_{T})}.
\]
Theorem \ref{samen} tells us that if we extend marked local complementation to
weighted graphs in the obvious way (i.e. local complementation does not affect
vertex weights), then the weighted, marked-graph bracket polynomial is
invariant under marked local complementations.

The simplest result of \cite{T3} is that reversing the $\alpha$ and $\beta$
weights of a vertex has the same effect as toggling its loop status. (The
corresponding $A-B$ duality is apparent in the formulas of Section 7.) This
observation may seem trivial but it is useful in simplifying recursive
calculations that involve the first equality of Proposition \ref{oldrec}:
rather than replace a graph with two graphs each time we want to remove a
loop, we simply interchange $\alpha(v)$ and $\beta(v)$ at each looped vertex.
Similarly, if the mark of a vertex $v$ includes $r$ then $[G]$ is unchanged if
we remove the $r$ and interchange $\alpha(v)$ and $\beta(v)$.

It is a simple matter to modify Theorem \ref{recursion} to incorporate
weights: just replace each occurrence of $A$ with $\alpha(v)$, and each
occurrence of $B$ with $\beta(v)$.

There are also analogues of series-parallel reductions, involving twin
vertices. (Recall from \cite{T3} that twin vertices occur naturally in the
looped interlacement graphs of link diagrams: when two strands of a link are
twisted around each other repeatedly, the resulting classical crossings
correspond to twin vertices in $\mathcal{L}(D,C)$.) Some of these twin
reductions are very much like series-parallel reductions; they replace several
vertices with one re-weighted vertex in a single graph. Others are not
completely analogous to series-parallel reductions as the reduced forms
involve two different graphs. All are of some value in computation because
they result in smaller graphs than part (d) of Theorem \ref{recursion}. There
are several different cases; here are four.

\begin{proposition}
Let $v\neq w\in V(G)$ be nonadjacent, unlooped twin vertices. (That is, they
have the same neighbors outside $\{v,w\}$.)

(a) Suppose that $v$ and $w$ are both marked $u$. Then $[G]=[(G-w)^{\prime}]
$, where $(G-w)^{\prime}$ is obtained from $G-w$ by changing the weights of
$v$ to $\alpha^{\prime}(v)=\alpha(v)\alpha(w)d+\alpha(v)\beta(w)+$
$\beta(v)\alpha(w)$ and $\beta^{\prime}(v)=\beta(v)\beta(w)$.

(b) Suppose that $v$ is unmarked and $w$ is marked $u$. Then
$[G]=[(G-w)^{\prime}]\,$, where $(G-w)^{\prime}$ is obtained from $G-w$ by
changing the weights of $v$ to $\alpha^{\prime}(v)=\alpha(v)\alpha
(w)d+\alpha(v)\beta(w)+$ $\beta(v)\alpha(w)$ and $\beta^{\prime}%
(v)=\beta(v)\beta(w)$.

(c) Suppose that $v$ and $w$ are both unmarked. Then $[G]=[(G-w)^{\prime}]$,
where $(G-w)^{\prime}$ is obtained from $G-w$ by giving $v$ a mark of $u$ and
changing the weights of $v$ to $\alpha^{\prime}(v)=\alpha(v)\alpha
(w)d+\alpha(v)\beta(w)+$ $\beta(v)\alpha(w)$ and $\beta^{\prime}%
(v)=\beta(v)\beta(w)$.

(d) Suppose that $v$ is marked $c$ and $w$ is marked $u$. Then
$[G]=[(G-w)^{\prime}]+\alpha(v)\beta(w)[G-v-w]$, where $(G-w)^{\prime}$ is
obtained from $G-w$ by unmarking $v$ and changing the weights of $v$ to
$\alpha^{\prime}(v)=\alpha(v)\alpha(w)+$ $\beta(v)\alpha(w)$ and
$\beta^{\prime}(v)=\beta(v)\beta(w)$.
\end{proposition}

\begin{proof}
(a) Suppose $V(G)=\{v_{1},...,v_{n}\}$ with $v=v_{1}$ and $w=v_{2}$. If
$T\subseteq\{v_{3},...,v_{n}\}\,$\ then the equality
\[
\nu%
\begin{pmatrix}
0 & 0 & \mathbf{1} & \mathbf{0}\\
0 & 0 & \mathbf{1} & \mathbf{0}\\
\mathbf{1} & \mathbf{1} & M_{11} & M_{12}\\
\mathbf{0} & \mathbf{0} & M_{21} & M_{22}%
\end{pmatrix}
-1=\nu%
\begin{pmatrix}
0 & \mathbf{1} & \mathbf{0}\\
\mathbf{1} & M_{11} & M_{12}\\
\mathbf{0} & M_{21} & M_{22}%
\end{pmatrix}
\]
tells us that $\nu(\mathcal{A}(G)_{T})-1=\nu(\mathcal{A}((G-w)^{\prime}%
)_{T})=\nu(\mathcal{A}(G)_{T\cup\{w\}})=\nu(\mathcal{A}(G)_{T\cup\{v\}}) $. It
follows that the contributions of $T$, $T\cup\{w\}$ and $T\cup\{v\}$ to $[G]$
sum to the contribution of $T$ to $[(G-w)^{\prime}]$. As
\[
\mathcal{A}(G)_{T\cup\{v,w\}}=%
\begin{pmatrix}
M_{11} & M_{12}\\
M_{21} & M_{22}%
\end{pmatrix}
=\mathcal{A}((G-w)^{\prime})_{T\cup\{v\}},
\]
the contribution of $T\cup\{v,w\}$ to $[G]$ coincides with the contribution of
$T\cup\{v\}$ to $[(G-w)^{\prime}]$.

(b) If $T\subseteq\{v_{3},...,v_{n}\}\,$\ then the equality
\[
\nu%
\begin{pmatrix}
0 & 0 & \mathbf{1} & \mathbf{0}\\
0 & 0 & \mathbf{1} & \mathbf{0}\\
\mathbf{1} & \mathbf{1} & M_{11} & M_{12}\\
\mathbf{0} & \mathbf{0} & M_{21} & M_{22}%
\end{pmatrix}
-1=\nu%
\begin{pmatrix}
0 & \mathbf{1} & \mathbf{0}\\
\mathbf{1} & M_{11} & M_{12}\\
\mathbf{0} & M_{21} & M_{22}%
\end{pmatrix}
=\nu%
\begin{pmatrix}
1 & 0 & \mathbf{1} & \mathbf{0}\\
0 & 0 & \mathbf{1} & \mathbf{0}\\
\mathbf{1} & \mathbf{1} & M_{11} & M_{12}\\
\mathbf{0} & \mathbf{0} & M_{21} & M_{22}%
\end{pmatrix}
\]
implies$\ \nu(\mathcal{A}(G)_{T})-1=\nu(\mathcal{A}((G-w)^{\prime})_{T}%
)=\nu(\mathcal{A}(G)_{T\cup\{w\}})=\nu(\mathcal{A}(G)_{T\cup\{v\}}) $. Also,%
\[
\mathcal{A}(G)_{T\cup\{v,w\}}=%
\begin{pmatrix}
1 & \mathbf{1} & \mathbf{0}\\
\mathbf{1} & M_{11} & M_{12}\\
\mathbf{0} & M_{21} & M_{22}%
\end{pmatrix}
=\mathcal{A}((G-w)^{\prime})_{T\cup\{v\}}.
\]

(c) The first equalities displayed for (b) still tell us that $\nu
(\mathcal{A}(G)_{T})-1=\nu(\mathcal{A}((G-w)^{\prime})_{T})=\nu(\mathcal{A}%
(G)_{T\cup\{w\}})=\nu(\mathcal{A}(G)_{T\cup\{v\}})$. In this case
$\mathcal{A}(G)_{T\cup\{v,w\}}\not =\mathcal{A(}(G-w)^{\prime})_{T\cup\{v\}} $
but $\nu(\mathcal{A}(G)_{T\cup\{v,w\}})=\nu(\mathcal{A(}(G-w)^{\prime}%
)_{T\cup\{v\}})$ nevertheless, because
\[
\nu%
\begin{pmatrix}
1 & 0 & \mathbf{1} & \mathbf{0}\\
0 & 1 & \mathbf{1} & \mathbf{0}\\
\mathbf{1} & \mathbf{1} & M_{11} & M_{12}\\
\mathbf{0} & \mathbf{0} & M_{21} & M_{22}%
\end{pmatrix}
=\nu%
\begin{pmatrix}
M_{11} & M_{12}\\
M_{21} & M_{22}%
\end{pmatrix}
.
\]

(d) This follows from the equality
\[
\nu%
\begin{pmatrix}
0 & \mathbf{1} & \mathbf{0}\\
\mathbf{1} & M_{11} & M_{12}\\
\mathbf{0} & M_{21} & M_{22}%
\end{pmatrix}
=\nu%
\begin{pmatrix}
1 & 0 & \mathbf{1} & \mathbf{0}\\
0 & 0 & \mathbf{1} & \mathbf{0}\\
\mathbf{1} & \mathbf{1} & M_{11} & M_{12}\\
\mathbf{0} & \mathbf{0} & M_{21} & M_{22}%
\end{pmatrix}
.
\]

\end{proof}

The similarities among cases (a), (b) and (c) are not coincidental: they
reflect the fact that the corresponding graphs are transformed into each other
by $r$-simplifications and marked local complementations at $v$ and $w$.

Conway \cite{C} introduced a valuable way to analyze a link diagram in terms
of smaller building blocks called \textit{tangles}. (According to Quach
Hongler and Weber \cite{QW} this notion dates back much further, but was
largely forgotten until Conway rediscovered it.) In \cite{T3} we showed that
tangles in a link diagram $D$ give rise to descriptions of $\mathcal{L}(D,C)$
as a \textit{composition} of graphs. This important construction is due to
Cunningham \cite{Cu}.

\begin{definition}
\label{comp} Let $F$ and $H$ be doubly marked, weighted graphs whose
intersection consists of a single unlooped, unmarked vertex $a$ with weights
$\alpha(a)=A$ and $\beta(a)=B$. The \emph{composition} $G=F\ast H$ is
constructed as follows.

(a) The elements of $V(G)=(V(F)\cup V(H))-\{a\}$ inherit their loops, marks
and weights from $F$ and $H$.

(b) $E(G)=E(F-a)\cup E(H-a)\cup\{\{v,w\}|\{v,a\}\in E(F)$ and $\{a,w\}\in
E(H)\}$.

(c) The number of free loops of $G$ is $\phi(G)=\phi(F)+\phi(H)$.
\end{definition}

The restrictions on $a$ are intended merely to ensure that no information is
lost when $a$ is removed in the construction; they have no effect on $F\ast H
$.

\begin{theorem}
\label{pjoin} Let $F$ be a marked, weighted graph with an unlooped, unmarked
vertex $a$ that has $\alpha(a)=A$ and $\beta(a)=B$. Then the ring $R$ has
elements $\alpha^{\prime}(a)$, $\beta^{\prime}(a)$ and $\gamma$ that depend
only on $F$ and $a$, and have the following \textquotedblleft
universal\textquotedblright\ property: every composition $F\ast H$ has
\[
\lbrack F\ast H]=[H^{\prime}]+\gamma\lbrack H-a],
\]
where $H^{\prime}$ is obtained from $H$ by changing the weights of $a$ from
$A$ and $B$ to $\alpha^{\prime}(a)$ and $\beta^{\prime}(a)$.
\end{theorem}

\begin{proof}
If $V(F)=\{a\}$ then $F\ast H=H-a$, so the theorem is satisfied with
$\alpha^{\prime}(a)=0=\beta^{\prime}(a)$ and $\gamma=1$.

We proceed using induction on the number of steps of Theorem \ref{recursion}
that may be applied within the subgraph $F-a$ of $F\ast H$. If $F$ has a free
loop and $F^{\prime}$ is obtained from $F$ by removing the free loop then part
(b) of Theorem \ref{recursion} tells us that the values of $\alpha^{\prime
}(a)$, $\beta^{\prime}(a)$ and $\gamma$ appropriate for $F$ are obtained from
those appropriate for $F^{\prime}$ by multiplying by $d$. Similarly, if $F$
has an isolated vertex $v$ then part (c) of Theorem \ref{recursion} tells us
that the values of $\alpha^{\prime}(a)$, $\beta^{\prime}(a)$ and $\gamma$
appropriate for $F$ are obtained from those appropriate for $F-v$ by
multiplying by $[\{v\}]$.

If $F$ has an unlooped vertex $v$ marked $c$ then part (d) of Theorem
\ref{recursion} tells us that $[F\ast H]=\alpha(v)[(F\ast H)-v]+\beta
(V)[(F\ast H)_{cru}^{v}-v]$. If $v$ is not a neighbor of $a$ in $F$ then we
conclude that
\[
\lbrack F\ast H]=\alpha(v)[(F-v)\ast H]+\beta(v)[(F_{cru}^{v}-v)\ast H],
\]
and hence the values of $\alpha^{\prime}(a)$, $\beta^{\prime}(a)$ and $\gamma$
appropriate for $F$ are obtained from the values appropriate for $F-v$ and
$F_{cru}^{v}-v$ by multiplying by $\alpha(v)$ and $\beta(v)$ respectively, and
then adding. If $v$ is a neighbor of $a$ in $F$ then we have
\[
\lbrack F\ast H]=\alpha(v)[(F-v)\ast H]+\beta(v)[(F_{cru}^{v}-v)\ast
H_{cru}^{a}].
\]
The inductive hypothesis tells us that we may express the first summand as
$[H^{\prime}]+\gamma_{1}[H-a]$ and the second as $[(H_{cru}^{a})^{\prime
\prime}]+\gamma_{2}[H_{cru}^{a}-a]$, where $(H_{cru}^{a})^{\prime\prime}$
differs from $H_{cru}^{a}$ only in the weights and loop-mark status of $a$.
The equalities (\ref{equalities}) of Section 7 tell us that $[(H_{cru}%
^{a})^{\prime\prime}]+\gamma_{2}[H_{cru}^{a}-a]$ may be incorporated into
$[H^{\prime}]+\gamma_{1}[H-a]$ by adding $\gamma_{2}$ to $\beta^{\prime}(a)$
and by adding each of $\alpha^{\prime\prime}(a),\beta^{\prime\prime}(a)$ to
$\alpha^{\prime}(a)$ or $\beta^{\prime}(a)$ or $\gamma_{1}$, as dictated by
the loop-mark status of $a$ in $H_{cru}^{a}$.

If $F$ has a looped vertex marked $cr$ the same argument applies. If $F$ has
an unlooped vertex marked $cr$ or a looped vertex marked $c$, simply reverse
the roles of $\alpha(v)$ and $\beta(v)$.

If $F$ has no vertex marked $c$ or $cr$ then we would like to apply part (e)
or part (f) of Theorem \ref{recursion} at some vertex $v$ or $w$ of
$V(F)-\{a\}$. If $v$ or $w$ is a neighbor of $a$ in $F$ then a side effect of
the local complementation will be to replace the subgraph $H-a$ of $F\ast H$
with $H_{cru}^{a}-a$, so the inductive hypothesis will give us an equality of
the form $[F\ast H]=[(H_{cru}^{a})^{\prime\prime}]+\gamma\lbrack H_{cru}%
^{a}-a]$ rather than $[F\ast H]=[H^{\prime}]+\gamma\lbrack H-a]$. As above,
$(H_{cru}^{a})^{\prime\prime}$ denotes a graph that differs from $H_{cru}^{a}$
only in the weights and loop-mark status of $a$, and the equalities
(\ref{equalities}) tell us how to transform $[F\ast H]=[(H_{cru}^{a}%
)^{\prime\prime}]+\gamma\lbrack H_{cru}^{a}-a]$ into a formula of the required
form $[F\ast H]=[H^{\prime}]+\gamma\lbrack H-a]$.
\end{proof}

The corresponding theorem of \cite{T3} has the additional hypothesis
\textquotedblleft no neighbor of $a$ in $H$ is marked,\textquotedblright\ but
as noted at the end of Section 7 this hypothesis is no longer necessary when
using marked local complementation. The conclusion was also phrased
differently in \cite{T3} -- the term $\gamma\lbrack H-a]$ was replaced by a
term $[H_{m}^{\prime}]$ involving a re-weighted version of $a$ marked $c$ --
but according to the formula for $G[(v,c)]$ given in (\ref{equalities}), that
phrasing is equivalent to the one here, as the re-weighted version of $a $ had
its $\beta$ weight equal to 0, and its $\alpha$ weight equal to $\gamma$.

\bigskip

\textbf{Acknowledgment} We are sincerely grateful to D. P. Ilyutko, V. O.
Manturov, L. Zulli and an anonymous referee for advice, encouragement and inspiration.


\bigskip

\begin{thebibliography}{99}                                                                                               %
\bibitem {A2}R. Arratia, B. Bollob\'{a}s and G. B. Sorkin, The interlace
polynomial of a graph, \textit{J. Combin. Theory Ser. B} \textbf{92} (2004) 199-233.

\bibitem {A}R. Arratia, B. Bollob\'{a}s and G. B. Sorkin, A two-variable
interlace polynomial, \textit{Combinatorica} \textbf{24} (2004) 567-584.

\bibitem {Bgiso}A. Bouchet, Graphic presentation of isotropic systems,
\textit{J. Combin. Theory Ser. B} \textbf{45} (1988) 58-76.

\bibitem {Bu}A. Bouchet, Unimodularity and circle graphs, \textit{Discrete
Math.} \textbf{66} (1987) 203-208.

\bibitem {B}A. Bouchet, Multimatroids III. Tightness and fundamental graphs,
\textit{Europ. J. Combin.} \textbf{22} (2001) 657-677.

\bibitem {Br}H. R. Brahana, Systems of circuits on two-dimensional manifolds,
\textit{Ann. Math. }\textbf{23} (1921) 144-168.

\bibitem {CL}M. Cohn and A. Lempel, Cycle decomposition by disjoint
transpositions, \textit{J. Combin. Theory Ser. A} \textbf{13} (1972) 83-89.

\bibitem {C}J. H. Conway, An enumeration of knots and links, and some of their
algebraic properties, in \textit{Computational Problems in Abstract Algebra},
Oxford, UK (1967) (Pergamon, 1970), pp. 329--358.

\bibitem {Cu}W. H. Cunningham, Decomposition of directed graphs, \textit{SIAM
J. Alg. Disc. Meth.} \textbf{3} (1982) 214-228.

\bibitem {EMS}J. A. Ellis-Monaghan and I. Sarmiento, Distance hereditary
graphs and the interlace polynomial. \textit{Combin. Probab. Comput.}
\textbf{16} (2007) 947-973.

\bibitem {FKM}R. Fenn, L. H. Kauffman and V. O. Manturov, Virtual knot theory
- unsolved problems, \textit{Fund. Math.} \textbf{188} (2005) 293--323.

\bibitem {GR}C. Godsil and G. Royle, \textit{Algebraic Graph Theory}, Graduate
Texts in\ Mathematics \textbf{207} (Springer-Verlag, Berlin-Heidelberg-New
York, 2001).

\bibitem {I}D. P. Ilyutko, An equivalence between the set of graph-knots and
the set of homotopy classes of looped graphs, preprint, arxiv: 1001.0360v1.

\bibitem {IM}D. P. Ilyutko and V. O. Manturov, Introduction to graph-link
theory, \textit{J. Knot Theory Ramifications} \textbf{18} (2009) 791-823.

\bibitem {IM1}D. P. Ilyutko and V. O. Manturov, Graph-links, preprint, arxiv: 1001.0384v1.

\bibitem {J1}F. Jaeger, On Tutte polynomials and cycles of plane graphs,
\textit{J. Combin. Theory Ser. B} \textbf{44} (1988) 127-146.

\bibitem {Jo}V. F. R. Jones, A polynomial invariant for links via von Neumann
algebras, \textit{Bull. Amer. Math. Soc.} \textbf{12} (1985) 103-112.

\bibitem {Kau}L. H. Kauffman, State models and the Jones polynomial,
\textit{Topology} \textbf{26 }(1987) 395-407.

\bibitem {Kv}L. H. Kauffman, Virtual knot theory, \textit{Europ. J.
Combinatorics} \textbf{20} (1999) 663-691.

\bibitem {Kd}L. H. Kauffman, Knot diagrammatics, in \textit{Handbook of Knot
Theory}, eds. W. Menasco and M. Thistlethwaite (Elsevier, Amsterdam, 2005).

\bibitem {K}A. Kotzig, Eulerian lines in finite 4-valent graphs and their
transformations, in \textit{Theory of Graphs}, Tihany (1966) (Academic Press,
New York, 1968), pp. 219--230.

\bibitem {L}M. Las Vergnas, Eulerian circuits of 4-valent graphs imbedded in
surfaces, in \textit{Algebraic methods in graph theory}, Szeged (1978),
Colloq. Math. Soc. J\'{a}nos Bolyai \textbf{25} (North-Holland, Amsterdam-New
York, 1981), pp. 451-477.

\bibitem {Ma}P. Martin, \textit{Enum\'{e}rations eul\'{e}riennes dans les
multigraphes et invariants de Tutte-Grothendieck}, Th\`{e}se, Grenoble (1977).

\bibitem {M}V. O. Manturov, A proof of V. A. Vassiliev's conjecture on the
planarity of singular links, \textit{Izv. Ross. Akad. Nauk Ser. Mat.}
\textbf{69} (2005) 169-178; translation, \textit{Izv. Math.} \textbf{69}
(2005) 1025-1033.

\bibitem {Me}B.\ Mellor, A few weight systems arising from intersection
graphs, \textit{Michigan Math. J.} \textbf{51} (2003) 509-536.

\bibitem {O}O.-P. \"{O}stlund, Invariants of knot diagrams and relations among
Reidemeister moves, \textit{J. Knot Theory Ramifications} \textbf{10} (2001) 1215-1227.

\bibitem {Pev}P. A. Pevzner, \textit{DNA physical mapping and alternating
Eulerian cycles in colored graphs}, Algorithmica \textbf{13} (1995), 77-105.

\bibitem {P}M. Polyak, Minimal sets of Reidemeister moves, preprint, arxiv: 0908:3127v2.

\bibitem {QW}C. V. Quach Hongler and C. Weber, Amphicheirals according to Tait
and Haseman, \textit{J. Knot Theory Ramifications} \textbf{17} (2008) 1387--1400.

\bibitem {RR}R. C. Read and P. Rosenstiehl, On the Gauss crossing problem, in:
\textit{Combinatorics} (Proc. Fifth Hungarian Colloq., Keszthely, 1976), Vol.
II, Colloq. Math. Soc. J\'{a}nos Bolyai, \textbf{18}, North-Holland,
Amsterdam-New York, 1978, ps. 843--876.

\bibitem {S}E.\ Soboleva, Vassiliev knot invariants coming from Lie algebras
and 4-invariants, \textit{J. Knot Theory Ramifications} \textbf{10} (2001) 161-169.

\bibitem {Th}M. B. Thistlethwaite, A spanning tree expansion of the Jones
polynomial, \textit{Topology} \textbf{26} (1987) 297-309.

\bibitem {Tb}L. Traldi,\ Binary nullity,\ Euler circuits and interlace
polynomials, \textit{Europ. J. Combinatorics}, to appear.

\bibitem {T1}L. Traldi, A bracket polynomial for graphs, II. Links,\ Euler
circuits and marked graphs, \textit{J. Knot Theory Ramifications} \textbf{19}
(2010) 547-586.

\bibitem {T3}L. Traldi, A bracket polynomial for graphs, III. Vertex weights,
\textit{J. Knot Theory Ramifications}, to appear.

\bibitem {TZ}L. Traldi and L.\ Zulli, A bracket polynomial for graphs, I,
\textit{J. Knot Theory Ramifications} \textbf{18} (2009) 1681-1709.

\bibitem {U}E. Ukkonen, \textit{Approximate string-matching with q-grams and
maximal matches}, Theoret. Comput. Sci. \textbf{92} (1992), 191-211.

\bibitem {Z}L.\ Zulli, A matrix for computing the Jones polynomial of a knot,
\textit{Topology} \textbf{34} (1995) 717-729.
\end{thebibliography}
\end{document}